\documentclass[12pt]{amsart}

\usepackage{amsmath,amsfonts,amssymb,amscd,verbatim,delarray,verbatim}
  \textwidth
12.5cm \oddsidemargin 1cm \evensidemargin 1cm \textheight 19.5cm

\newcommand{\F}{{\mathbb F}}

\newcommand{\Q}{{\mathbb Q}}

\newcommand{\ord}{\operatorname{ord}}
\newcommand{\di}{\operatorname{diag}}

\begin{document}
\numberwithin{equation}{section}

\newtheorem{theorem}{Theorem}[section]
\newtheorem{lemma}[theorem]{Lemma}

\newtheorem{prop}[theorem]{Proposition}
\newtheorem{proposition}[theorem]{Proposition}
\newtheorem{corollary}[theorem]{Corollary}
\newtheorem{corol}[theorem]{Corollary}
\newtheorem{conj}[theorem]{Conjecture}
\newtheorem{sublemma}[theorem]{Sublemma}

\theoremstyle{definition}
\newtheorem{defn}[theorem]{Definition}
\newtheorem{example}[theorem]{Example}
\newtheorem{examples}[theorem]{Examples}
\newtheorem{remarks}[theorem]{Remarks}
\newtheorem{remark}[theorem]{Remark}
\newtheorem{algorithm}[theorem]{Algorithm}
\newtheorem{question}[theorem]{Question}
\newtheorem{problem}[theorem]{Problem}
\newtheorem{subsec}[theorem]{}
\newtheorem{clai}[theorem]{Claim}

\def\toeq{{\stackrel{\sim}{\longrightarrow}}}
\def\into{{\hookrightarrow}}


\def\alp{{\alpha}}  \def\bet{{\beta}} \def\gam{{\gamma}}
 \def\del{{\delta}}
\def\eps{{\varepsilon}}
\def\kap{{\kappa}}                   \def\Chi{\text{X}}
\def\lam{{\lambda}}
 \def\sig{{\sigma}}  \def\vphi{{\varphi}} \def\om{{\omega}}
\def\Gam{{\Gamma}}   \def\Del{{\Delta}}
\def\Sig{{\Sigma}}   \def\Om{{\Omega}}
\def\ups{{\upsilon}}


\def\F{{\mathbb{F}}}
\def\BF{{\mathbb{F}}}
\def\BN{{\mathbb{N}}}
\def\Q{{\mathbb{Q}}}
\def\Ql{{\overline{\Q }_{\ell }}}
\def\CC{{\mathbb{C}}}
\def\R{{\mathbb R}}
\def\V{{\mathbf V}}
\def\D{{\mathbf D}}
\def\BZ{{\mathbb Z}}
\def\K{{\mathbf K}}
\def\XX{\mathbf{X}^*}
\def\xx{\mathbf{X}_*}

\def\AA{\Bbb A}
\def\BA{\mathbb A}
\def\HH{\mathbb H}
\def\PP{\Bbb P}

\def\Gm{{{\mathbb G}_{\textrm{m}}}}
\def\Gmk{{{\mathbb G}_{\textrm m,k}}}
\def\GmL{{\mathbb G_{{\textrm m},L}}}
\def\Ga{{{\mathbb G}_a}}

\def\Fb{{\overline{\F }}}
\def\Kb{{\overline K}}
\def\Yb{{\overline Y}}
\def\Xb{{\overline X}}
\def\Tb{{\overline T}}
\def\Bb{{\overline B}}
\def\Gb{{\bar{G}}}
\def\Ub{{\overline U}}
\def\Vb{{\overline V}}
\def\Hb{{\bar{H}}}
\def\kb{{\bar{k}}}

\def\Th{{\hat T}}
\def\Bh{{\hat B}}
\def\Gh{{\hat G}}

\def\cF{{\mathfrak{F}}}
\def\cC{{\mathcal C}}
\def\cU{{\mathcal U}}

\def\Xt{{\widetilde X}}
\def\Gt{{\widetilde G}}

\def\gg{{\mathfrak g}}
\def\hh{{\mathfrak h}}
\def\lie{\mathfrak a}

\def\XX{\mathfrak X}
\def\RR{\mathfrak R}
\def\NN{\mathfrak N}

\def\minus{^{-1}}

\def\GL{\textrm{GL}}            \def\Stab{\textrm{Stab}}
\def\Gal{\textrm{Gal}}          \def\Aut{\textrm{Aut\,}}
\def\Lie{\textrm{Lie\,}}        \def\Ext{\textrm{Ext}}
\def\PSL{\textrm{PSL}}          \def\SL{\textrm{SL}}
\def\loc{\textrm{loc}}
\def\coker{\textrm{coker\,}}    \def\Hom{\textrm{Hom}}
\def\im{\textrm{im\,}}           \def\int{\textrm{int}}
\def\inv{\textrm{inv}}           \def\can{\textrm{can}}
\def\id{\textrm{id}}              \def\Char{\textrm{char}}
\def\Cl{\textrm{Cl}}
\def\Sz{\textrm{Sz}}
\def\ad{\textrm{ad\,}}
\def\SU{\textrm{SU}}
\def\Sp{\textrm{Sp}}
\def\PSL{\textrm{PSL}}
\def\PSU{\textrm{PSU}}
\def\rk{\textrm{rk}}
\def\PGL{\textrm{PGL}}
\def\Ker{\textrm{Ker}}
\def\Ob{\textrm{Ob}}
\def\Var{\textrm{Var}}
\def\poSet{\textrm{poSet}}
\def\Al{\textrm{Al}}
\def\Int{\textrm{Int}}
\def\Smg{\textrm{Smg}}
\def\ISmg{\textrm{ISmg}}
\def\Ass{\textrm{Ass}}
\def\Grp{\textrm{Grp}}
\def\Com{\textrm{Com}}
\def\rank{\textrm{rank}}

\def\char{\textrm{char}}

\newcommand{\Or}{\operatorname{O}}

\def\tors{_\def{\textrm{tors}}}      \def\tor{^{\textrm{tor}}}
\def\red{^{\textrm{red}}}         \def\nt{^{\textrm{ssu}}}

\def\sss{^{\textrm{ss}}}          \def\uu{^{\textrm{u}}}
\def\mm{^{\textrm{m}}}
\def\tm{^\times}                  \def\mult{^{\textrm{mult}}}

\def\uss{^{\textrm{ssu}}}         \def\ssu{^{\textrm{ssu}}}
\def\comp{_{\textrm{c}}}
\def\ab{_{\textrm{ab}}}

\def\et{_{\textrm{\'et}}}
\def\nr{_{\textrm{nr}}}

\def\nil{_{\textrm{nil}}}
\def\sol{_{\textrm{sol}}}
\def\End{\textrm{End\,}}

\def\til{\;\widetilde{}\;}

\def\min{{}^{-1}}

\def\AGL{{\mathbb G\mathbb L}}
\def\ASL{{\mathbb S\mathbb L}}
\def\ASU{{\mathbb S\mathbb U}}
\def\AU{{\mathbb U}}





\title[From Thompson to Baer--Suzuki]
{{\bf From Thompson to Baer--Suzuki: a sharp characterization of the
solvable radical}}

\author[Gordeev, Grunewald, Kunyavskii,  Plotkin] {Nikolai
Gordeev, Fritz Grunewald,  Boris Kunyavskii, Eugene Plotkin }
\address{Gordeev: Department of Mathematics, Herzen State
Pedagogical University,
48 Moika Embankment, 191186, St.Petersburg, RUSSIA} \email{nickgordeev@mail.ru}

\address{Grunewald: Mathematisches Institut der
Heinrich-Heine-Universit\"at D\"usseldorf, Universit\"atsstr. 1, 40225
D\"usseldorf, GERMANY} \email{grunewald@math.uni-duesseldorf.de}

\address{ Kunyavskii: Department of
Mathematics, Bar-Ilan University, 52900 Ramat Gan, ISRAEL}
\email{kunyav@macs.biu.ac.il}
\address{ Plotkin: Department of
Mathematics, Bar-Ilan University, 52900 Ramat Gan, ISRAEL}
\email{plotkin@macs.biu.ac.il}


\begin{abstract}
We prove that an element $g$ of prime order $ >3$ belongs to the
solvable radical $\RR (G)$ of a finite (or, more generally, a
linear) group if and only if for every $x\in G$ the subgroup
generated by $g, xgx^{-1}$ is solvable. This theorem implies that  a
finite (or a linear) group $G$ is solvable if and only if in each
conjugacy class of $G$ every two elements generate a solvable
subgroup.
\end{abstract}

\maketitle



\section{Introduction} \label{sec:intro}

The classical Baer--Suzuki theorem \cite{Ba}, \cite{Su2}, \cite{AL}
states that

\begin{theorem} [Baer--Suzuki] \label{th:BS}
The nilpotent radical of a finite group $G$ coincides with the
collection of $g\in G$ satisfying the property: for every $x\in G$
the subgroup generated by $g$ and $xgx\min$ is nilpotent.
\end{theorem}

Within past few years a lot of efforts have been made in order to
describe the solvable radical of a finite group and to establish a
sharp analogue of the Baer--Suzuki theorem with respect to the
solvability property (see \cite{Fl2}, \cite{Fl3}, \cite{GGKP1},
\cite{GGKP2}). In particular, the following problem is parallel to
the Baer--Suzuki result:

\begin{problem} \label{pr:1}
{\rm Let $G$ be a finite group with the solvable radical $\RR(G)$.
What is the minimal number $k$ such that $g\in \RR(G)$ if and only
if the subgroup generated by $x_1gx_1^{-1},\dots, x_kgx_k^{-1}$ is
solvable for every $x_1,\dots, x_k$ in $G$?}
\end{problem}

Recently (see \cite{GGKP3}) it was proved that

\begin{theorem} \label{th:main}
The solvable radical of a finite group $G$ coincides with the
collection of $g\in G$ satisfying the property: for every three
elements $a, b, c \in G$ the subgroup generated by the conjugates
$g,  aga\min,  bgb\min,  cgc\min$ is solvable.
\end{theorem}

Theorem \ref{th:main} is sharp: 
in the symmetric groups $S_n$ $(n\geq 5)$ every triple of
transpositions generates a solvable subgroup.

However, as mentioned by Flavell \cite{Fl2}, one can expect a
precise analogue of the Baer--Suzuki theorem to hold for the
elements of prime order greater than 3 in $\RR (G)$. Our main result
confirms this expectation:

\begin{theorem} \label{th:bigorder}
Let $G$ be a finite group. An element $g$ of prime order $\ell >3$
belongs to $\RR (G)$ if and only if for every $x\in G$ the subgroup
$H=\left< g,xgx\min\right>$ is solvable.
\end{theorem}

Theorem \ref{th:bigorder} together with Burnside's ${p^\alpha
q^\beta}$-theorem  implies

\begin{corol} \label{corol:solv}
A finite group $G$ is solvable if and only if in each conjugacy
class of $G$ every two elements generate a solvable subgroup.
\end{corol}

\begin{remark}
{\rm A standard argument (cf. \cite[Theorem 4.1]{GKPS},
\cite[Theorem 1.4]{GGKP3}) shows that Theorem \ref{th:bigorder} and
Corollary \ref{corol:solv} remain true for the linear groups (not
necessarily finite).}
\end{remark}

\begin{remark}
{\rm Corollary \ref{corol:solv} can be viewed as an extension of a
theorem of J.~Thompson \cite{Th}, \cite{Fl1} which states that  a
finite group $G$ is solvable if and only if every two-generated
subgroup of $G$ is solvable.}
\end{remark}

\begin{remark}
{\rm The proof of Theorem \ref {th:bigorder} uses the classification
of finite simple groups (CFSG). The proof of Corollary
\ref{corol:solv} can be obtained without classification using the
above mentioned J.~Thompson's characterization of the minimal
non-solvable groups. Flavell managed to prove, without CFSG, an
analogue of Theorem \ref{th:main} for $k=10$ \cite{Fl2} and Theorem
\ref{th:main} under the additional assumption that $g\in G$ is of
prime order $\ell >3$.}
\end{remark}

\begin{remark}
{\rm R.~Guralnick informed us that Theorems \ref{th:bigorder} and
\ref{th:main} were independently proved in forthcoming works by
Guest, Guralnick, and Flavell  \cite{Gu}, \cite{FGG}. Flavell
\cite{FGG} reduced $k$ in Problem \ref{pr:1} to 7 with a proof which
does not rely on CFSG.}
\end{remark}

\begin{remark}
The problem of explicit description of the solvable radical of a
finite group in terms of quasi-Engel sequences (see \cite{BBGKP},
\cite{GPS}) is still open: there is no explicit analogue of Baer's
theorem on characterizing the nilpotent radical as the collection of
Engel elements. However, a recent result by J.~S.~Wilson \cite{Wi},
stating the existence of a countable set of words in two variables
(in spirit of \cite{BW}) which can be used to describe the solvable
radical, gives much hope for such a characterization.
\end{remark}


The results of the present paper were announced in \cite{GGKP4}.


\noindent{\it Notational conventions}. Whenever possible, we
maintain the notation of \cite{GGKP2} which mainly follows
\cite{St1}, \cite{Ca2}. Let $G(\Phi,K)$ be a Chevalley group where
$\Phi$ is a reduced irreducible root system and $K$ is a field.
Denote by $W=W(\Phi)$ the Weyl group corresponding to $\Phi$. Denote
by $\dot w$ a preimage of $w\in W$ in $G(\Phi,K)$. Twisted Chevalley
groups and Suzuki and Ree groups are denoted by $^dG(\Phi,K),$
$d=2,3$. We call Chevalley groups (twisted, untwisted, Suzuki and
Ree groups) the groups of Lie type. Chevalley groups $G(\Phi,K)$ are
denoted throughout the paper mostly as groups of type $\Phi(K)$.
Correspondingly, for finite fields $K=\mathbb F_{q}$, $q=p^n$, they
are denoted just by $\Phi(q).$ We adopt the notation of \cite{Ca2}
for twisted Chevalley groups which means that we use the symbols
$^2\Phi(q^2)$ but not $^2\Phi(q)$. For example, simple unitary
groups are denoted either as $^2A_n(q^2)$, or as  $PSU_n(q^2)$  (and
not by $PSU_n(q)$), or as $PSU_n(F)$, where $F$ is a quadratic
extension of $K$. We use the same notation for Suzuki and Ree groups
(this means that in these cases $q$ is not integer because $q^2$ is
an odd power of 2 or 3).

We use the standard notation $u_\alpha(t)$, $\alpha\in \Phi$, $t\in
K$, for elementary unipotent elements of $G$. Correspondingly, split
semisimple elements will be denoted by $h_\alpha(t),$ $t\in K^*$,
where $K^*$ is the multiplicative group of $K$.


We say that a finite group $G$ is almost simple if it has a unique
normal simple subgroup $L$ such that $L\leq G\leq \Aut (L)$.  In the
classification of automorphisms we follow \cite[p.~60]{GLS},
\cite[p.~78]{GL}. This means that all automorphisms of an adjoint
 group of Lie type are subdivided to inner-diagonal automorphisms,
field automorphisms, graph automorphisms, and graph-field ones (for
non-adjoint groups see \cite[p.~79]{GL}). Recall that according to
\cite[Definition~2.5.13]{GLS}, any field automorphism of prime order
$\ell>3$ of $L$ is conjugate in $\Aut (L)$ to a standard one in the
sense of \cite{St1}.

We use the formula $[x,y]=xyx\min y\min $ to denote the commutator.
If $H$ is a subgroup of $G$, we denote by $H^a$, $a\in G$, the
subgroup $aHa^{-1}$. For the group of fixed points of an
automorphism $a$ of a group $H$ we use the centralizer notation
$C_H(a)$ (both for inner and outer automorphisms of $H$). The only
exception is the symbol $\mathbb{G}^F$, which is traditionally used
for denoting the group of fixed points of a simple algebraic group
with respect to a Frobenius endomorphism (see \cite{Ca2}).

We use below some standard language of algebraic groups ([Sp],
[Hu]). Here we consider only algebraic groups defined over a finite
field $\mathbb F_q$ and therefore sometimes we identify such groups
with the groups of points over the algebraic closure
$\overline{\mathbb F}_q$. By a Chevalley group we mean here the
group of points of a reductive algebraic group which is defined and
quasisplit over $K$. Note that all groups are quasisplit over finite
fields.


\section{Strategy of proof} \label{sec:strat}

As in \cite{GGKP1}--\cite{GGKP3}, we reduce Theorem
\ref{th:bigorder} to the following statement:

\begin{theorem} \label{th:radel}
Let $G$ be a finite almost simple group, and let $g\in G$ be of
prime order $> 3$. Then there is $x\in G$ such that the subgroup
generated by $g$ and $ xgx\min$ is not solvable.
\end{theorem}

Although this reduction is fairly standard, we sketch its main steps
below. Let $S(G)$ be the set of all elements  $g\in G$ of prime
order bigger than 3 such that for every $x\in G$ the subgroup
$\langle g, xgx\min \rangle$ is solvable.

Obviously, any element of $\RR(G)$ of prime order $> 3$ lies in
$S(G)$, and we have to prove the opposite inclusion. We may assume
that $G$ is semisimple (i.e., $\RR(G)=1$), and we shall prove that
$G$ does not contain elements from $S(G)$. Assume the
contrary and consider a minimal counterexample (i.e. a semisimple
group $G$ of smallest order with $S(G)\ne \emptyset$).

It is easy to see that any $g\in G$ acts as an automorphism (denoted
by the same letter $g$) on the CR-radical $V$ of $G$ (see
\cite[3.3.16]{Ro}) and that  $V=H_1\times \cdots\times H_n$ where
all $H_i,$ $1\leq i\leq n$, are isomorphic (say to $H$) nonabelian
simple groups (\cite[Section~2]{GGKP2}). Suppose that $g\neq 1$
belongs to $S(G)$. Let us show that $g$ cannot act on $V$ as a
non-identity element of the symmetric group $S_n$.

Since $g\in S(G)$, the subgroup $\Gamma=\langle g,  xgx\min\rangle$
is solvable for any $x\in G$. Take $x\in V$. Evidently, $\Gamma$
contains the elements $[g,x]= gxg\min x\min=g(x)x\min $ and
$g^2(x)x\min .$ Denote by $\sigma$ the element of $S_n$
corresponding to $g$.

Suppose $\sigma\neq 1$.  Since the order of $\sigma$ is greater than
$3$, we may assume that there exist $i$ and  $j$ such that
$\sigma(j)=i$ and $\sigma(i)=1$. Take $ x=(x_1,\ldots,x_n) \in
V$ such that $x_j=b$, $x_i=a$, where $a$ and $b$ are generators of
the simple group $H$ and $x_k=1$ for $k\neq i,j$, $1\leq k\leq n$.
Then the group $\langle g(x )x\min, g^2(x )x\min\rangle$ is not
solvable since $(g(x )x\min)_1=a$ and $(g^2(x )x\min)_1=b$ and these
elements generate a simple group $H$. Contradiction with the
assumption that $\Gamma$ is solvable.

So we may assume that an element $g\in S(G)$ acts as an
automorphism $\tilde g$ of the simple group $H$. Then we consider
the extension of $H$ by $\tilde g$. Denote this almost simple group
by $G_1$. By Theorem \ref{th:radel}, $G_1$ contains no
elements from $S(G)$. Contradiction with the choice of $\tilde g$.


So the rest of the paper is devoted to the proof of Theorem
\ref{th:radel}. We refer to the property stated in Theorem
\ref{th:radel} as Property {\bf{(NS)}} (for ``non-solvable"):

\medskip

{\bf{(NS)} For every $g\in G$ of prime order $> 3$ there is $x\in G$
such that the subgroup generated by $g$ and $x gx\min$ is not
solvable.}

\medskip

We use CFSG to prove, by case-by-case analysis, that every almost
simple group satisfies {\bf{(NS)}}. Section \ref{sec:spor} deals
with alternating and sporadic groups. In Section \ref{sec:rank1} we
consider groups of Lie type of rank 1. 
In Section \ref{sec:gen} the general case is treated. Finally, the
exceptional case ${}^2F_4$ is treated separately in Section
\ref{sec:F4}.

\section{Alternating, symmetric, and sporadic groups}\label{sec:spor}

Let $G$ be an almost simple group, $L\leq G\leq \Aut L$.

\begin{lemma} \label{lm:sym}
Let $L=A_n$, $n\ge 5$, be the alternating group on $n$ letters. Then
$G$ satisfies $\bf{(NS)}$.
\end{lemma}

\begin{proof}
Clearly it is enough to consider the alternating groups: as $\Aut
(A_n)=S_n$ for $n\ne 6$ and $[\Aut (A_6):A_6]=4$, any element of odd
order in $\Aut (A_n)$ lies in $A_n$. So let $G=A_n$, $n\ge 5$. For
$n=5$ the proof is straightforward, so we may proceed by induction.
We may thus assume that $g$ acts without fixed points, so $n=k\ell$,
where $\ell$ stands for the order of $g$, and $g$ is a product of
$k$ disjoint cycles of length $\ell$. If $k=1$, we can conjugate
$g=(12\dots \ell)$ by a 3-cycle $z=(123)$ to see that $\left<g,z
gz\min\right> =A_{\ell}$. For $k>1$, we conjugate $g$ by a product of
$k$ 3-cycles.
\end{proof}

\begin{lemma} \label{th:sym}
Let $L$  be a sporadic simple group. Then $G$ satisfies $\bf{(NS)}$.
\end{lemma}

\begin{proof}  As the group of outer automorphisms of any sporadic
group is of order at most 2, it is enough to treat the case where
$G$ is a simple sporadic group. Here the proof goes, word for word,
as in \cite[Prop.~9.1]{GGKP1}. Namely, case-by-case analysis shows
that any element $g\in G$ of prime order $\ell >3$ is either
contained in a smaller simple subgroup of $G$, or its normalizer is
a maximal subgroup of $G$. In the latter case it is enough to
conjugate $g$ by an element $x$ not belonging to
$N_G(\left<g\right>)$ to ensure that $\left<g, x gx\min\right>=G$.
\end{proof}


\section{Groups of Lie rank 1}\label{sec:rank1}

In this case our proof combines arguments of several different
types. In the cases $L=PSL_2(q)$ and $L=PSU_3(q)$ we use the
analysis of \cite{GS} with appropriate modifications whenever
needed. The case of inner automorphisms of Suzuki and Ree groups is
treated in the same spirit as in \cite{GGKP1} (see Proposition
\ref{lem-Suzuki-Ree}). The case of field automorphisms of Ree groups
can be reduced to the $PSL_2$-case. Finally, in the case of field
automorphisms of Suzuki groups we apply a counting argument similar
to \cite{GS}.


Before starting the proof, let us make some preparations. The
following fact is well known.

\begin{prop}\label{prop:aut_prime}
Let $G$ be a finite almost simple group of Lie type, and let $g\in
G$ be an element of prime order $\ell>3$. Then $g$ is either an
inner-diagonal or a field automorphism of $L$.
\end{prop}

\begin{proof} See \cite[p.~82, 7-3]{GL} and \cite[Proposition 1.1]{LLS}.
\end{proof}

\begin{prop}\label{prop:small}
Let $L$ be one of the following groups: $^2A_2(9)$, $^2G_2(3)$,
$A_2(2)$, $A_2(3)$, $B_2(2)$, $B_2(3)$, $G_2(2)$, $G_2(3)$,
$^2A_3(9)$, $^2A_4(9)$,  $^3D_4(8)$, $^3D_4(27)$, $^2F_4(2)$. Then
$G$ satisfies $\bf{(NS)}$.
\end{prop}

\begin{proof}
We use \cite[Table 1]{GGKP2} and straightforward MAGMA computations
with outer automorphisms of $L$.
\end{proof}

\begin{remark}
If a group $L$ from the above list is not simple, the computations
have been made for its derived subgroup $L'$ which is simple.
\end{remark}

So from now on we can exclude the groups listed in Proposition
\ref{prop:small} from the further considerations.

Recall now, for the reader's convenience, a theorem of Gow
\cite{Gow} which is essential in our argument.

{\it Let $L$ be a finite simple group of Lie type, and let $z\neq 1$
be a semisimple element in $L$. Let $C$ be a conjugacy class of $L$
consisting of regular semisimple elements. Then there exist $g\in C$
and $x\in L$ such that $z=[g,x]$.}

\begin{theorem} \label{prop:rank1}
Suppose that the Lie rank of $L$ is 1. Then $G$ satisfies
$\bf{(NS)}$.
\end{theorem}

\noindent{\it Proof}. Let $g\in G$ be of prime order $> 3$. We shall
check that there is $x\in L$  such that the subgroup of $G$
generated by $g$ and $x gx\min$ is not solvable.

\begin{prop} \label{lem-PSL-PSU}
If $L=PSL_2(q)$, $q\ge 4$, or $L=PSU_3(q^2)$, $q>2$, then $G$
satisfies $\bf{(NS)}$.
\end{prop}

\begin{proof}
In the case $L=PSL_2(q)$ the result follows from \cite[Lemma
3.1]{GS}. If $L=PSU_3(q^2)$, $q>2$, the result follows from the
proof of \cite[Lemma 3.3]{GS} with the single exception when $g$ is
a field automorphism. In the latter case we can take $g$ to be
standard. The order of $g$ is a prime number bigger than 3, and we
may thus assume that $q\neq 2,3$ and that $g$ normalizes but does
not centralize a subgroup of type $A_1$ generated by a
self-conjugate root of $A_2$. The result follows from \cite[Lemma
3.1]{GS}.
\end{proof}

Let $L$ be a Suzuki group ${}^2B_2(q^2)$ or a Ree group
${}^2G_2(q^2)$ where $q^2$ is an odd power of 2, in the Suzuki case,
or of 3, in the Ree case (see, e.g., \cite{Su1}, \cite{Kl2},
\cite{LN}). Then $L = \mathbb{G}^F$ where $\mathbb{G}$ is the
corresponding simple algebraic group (of type $B_2$ or $G_2$)
defined over the field ${\mathbb F}_{2}$  or ${\mathbb F}_3$, and
$F$ is the appropriate Frobenius endomorphism of $\mathbb G$
(\cite{SS}, \cite[1.3,~20.1]{Hu2}). There exists an $F$-stable Borel
subgroup $\mathbb{B}\leq \mathbb{G}$. The group $\mathbb{B}^F$ will
be called below a Borel subgroup of $L$. We fix one of such
subgroups $B$. Every Borel subgroup of $L$ is of the form $B^a$ for
some $a \in L$. We will denote by $T$ a maximal subgroup of
semisimple elements of a Borel subgroup $B^a$. Note that $T$ is the
subgroup of $F$-fixed elements of an $F$-stable torus of
$\mathbb{G}$ contained in an $F$-stable Borel subgroup of
$\mathbb{G}$. Hence we will call such a group $T$ a {\it quasisplit
torus} of $L$. Furthermore, we denote by $\mathfrak{T}$ any group of
$F$-fixed elements of an $F$-stable torus of $\mathbb{G}$ which is
not contained in any $F$-stable Borel subgroup of $\mathbb{G}$. We
call such a group a {\it nonsplit torus} of $L$. Note that
$\mathfrak{T}\cap B^a = 1$ for every $a \in L$. Recall that all
maximal tori in Suzuki--Ree groups are cyclic (see \cite{Su1},
\cite{V}).

For Suzuki and Ree groups, we consider the cases of inner and outer
automorphisms separately. Since all diagonal automorphisms are inner
in Suzuki--Ree groups \cite{CCNPW}, the case of inner-diagonal
automorphisms is reduced to the case of inner ones. We start with
the case where $g$ is an inner automorphism.

\begin{prop} \label{lem-Suzuki-Ree}
If $L$ is a Suzuki group $^2B_2(q^2)$, $q^2=2^{2m+1}$, $m\geq 1$, or
a Ree group $^2G_2(q^2)$, $q^2=3^{2m+1},$ $m \ge 1$, and $g\in L$ is
of prime order $> 3$, then there exists $x \in L$ such that the
group $\left<g,x gx\min\right>$ is not solvable.
\end{prop}

\begin{proof}
As the order of $g$ is greater than 3, it cannot be unipotent, so we
may and shall assume that $g$ is semisimple. We argue as in
\cite[Section 4]{GGKP1}. Note that all tori in the Suzuki and Ree
groups are cyclic, and all semisimple elements of order greater than
3 are regular \cite{Su1}, \cite{Kl2}, \cite{LN}. By Gow's theorem,
for every semisimple element $z^\prime \in L$ we can find $x, y\in L$ such that
$z =yz^\prime y^{-1} =  [g,x]$. Consider two cases:

\begin{itemize}
\item $g$ is a generator of some maximal quasisplit torus; 
\item $g$ is not a generator of any maximal quasisplit torus.
\end{itemize}

In the first case, choose $x$ so that $z=[g,x]$ would be a generator
of a nonsplit torus. In the second case, choose $x$ so that
$z=[g,x]$ would be a generator of some quasisplit torus.

Note that in both cases $g\notin \langle z\rangle$.

With such a choice of $x$, let $H = \left<g,xgx\min\right>$. By
construction, we have $T \leq H$ for some quasisplit torus $T$.

First note that $H$ is not contained in $N_L(T)$. Indeed, $g$ and
$z$ cannot both normalize $T$ since they are of prime order $>3$ and
one of them does not belong to $T$, whereas the order of $N_L(T)/T$
is 2.

Furthermore, $H$ is not contained in any Borel subgroup. Indeed, if
both $g$ and $xgx^{-1}$ belong to a Borel subgroup $B'=T'U'$ (where
$T'$ is a fixed maximal quasisplit torus and $U'$ is the subgroup of
unipotent elements), then  we are in the second case. Consider the
cyclic group $B'/U'$. Let $\bar g$ and $\overline{x gx\min}=\bar
g_1$ be the corresponding images. Then $\bar g$ and $\bar g_1$ are
of the same order, $\left<\bar g\right>\neq T'$, $\left<\bar
g_1\right>\neq T'$, but $\left< \bar g^{-1}\bar g_1\right>\cong T'$.
Contradiction.

The Suzuki groups have no maximal subgroups other than $N_L(T)$,
$N_L(\mathfrak{T})$, $B$, and  Suzuki groups over smaller fields
\cite{Su1}. The subgroup $H$ is not contained in a subgroup of the
latter type since $H$ contains a maximal torus of $L$. Furthermore,
$H$ is not contained in $N_L(\mathfrak{T})$ since it contains a
quasisplit torus. So we conclude that $H=L$. Using similar arguments
and the list of the maximal subgroups of Ree groups \cite{Kl2},
\cite{LN}, one can show that $H$ lies in a maximal subgroup of a Ree
group only in two cases: either $q^2=3$ (which is excluded by
Proposition \ref{prop:small}), or $H<PSL_2(q^2)$. In the latter case
we use Proposition \ref{lem-PSL-PSU} to conclude that
$H=PSL_2(q^2)$.
\end{proof}


It remains to consider the case of outer automorphisms of prime
order of Suzuki and Ree groups. Any such automorphism is a field
automorphism (see Proposition \ref{prop:aut_prime}) which is assumed
to be standard.

We start with the simpler case of Ree groups.

\begin{prop} \label{lem-Ree-field}
Let $L$ be a Ree group $^2G_2(q^2),$ $q^2=3^{2m+1}$, $m\ge 1$, and
let $g$ be a field automorphism of $L$ of prime order $\ell > 3$.
Then there exists $x\in L$ such that $\left<g,x gx\min\right>$ is
not solvable.
\end{prop}

\begin{proof}
The group $L$ contains a subgroup isomorphic to $PSL_2(K)$,
$K=\mathbb F_{q^2}$. This subgroup is generated by the elementary
unipotent elements of $L$ of type $u_A(t), u_{-A}(t)$, $t\in K$,
where $u_A(t)=u_{a+b}(t^\vartheta)u_{3a+b}(t), $ $\vartheta\colon
K\to K$, $3\vartheta^2=1$, and  $a$, $b$ are the short and long
simple roots of $G_2$, respectively (see \cite{LN}).

Hence $g$ normalizes and does not centralize $PSL_2(K)$. Thus the
assertion of the proposition follows from Proposition
\ref{lem-PSL-PSU}.
\end{proof}

\begin{prop}\label{lemma:field}
Let $L$ be a Suzuki group $^2B_2(q^{2\ell})$, $q^2=2^{2m+1}$, $m\geq
0$, and let $g$ be a field automorphism of $L$ of prime order $\ell$
greater than 3. Then there is $x\in L$ such that the subgroup
$\left<g,xgx\min \right>\cap L$ is not solvable.
\end{prop}

\noindent {\it Proof}. Denote by $\Gamma$ the set of all $y=x
gx\min,$ $x\in L$,  such that the group $\Gamma_y:=\left<g,y\right> \cap
L$ is solvable. We shall prove that $|\Gamma|<|\{aga\min \mid a\in
L\}|$.

Note that $\Gamma_y$ is invariant under the action of $g$ because
$g\in\left<g,y\right>$, $gLg\min=L$, and
$\Gamma_y=\left<g,y\right>\cap L$.

Fix a Borel subgroup $B<L$, a maximal quasisplit torus $T<B$, and
maximal nonsplit tori $\mathfrak{T}$ of $L$  which are invariant
under the action of $g$. It is known that either
$\mathfrak{T}={\mathfrak{T}_1}$ or $\mathfrak{T}={\mathfrak{T}_2}$,
where the orders of cyclic groups ${\mathfrak{T}_1}$,
${\mathfrak{T}_2}$ are $q^{2\ell}+\sqrt{2q^{2\ell}}+1$ and $q^{2\ell}-\sqrt{2q^{2\ell}}+1$
respectively  (see \cite{Su1}, \cite{SS}). Recall that every maximal
subgroup of $L$ is conjugate to $B$, $N_L(T)$, $N_L(\mathfrak T)$,
or is isomorphic to a Suzuki group over a smaller field.

For every $y\in \Gamma$ the group $\Gamma_y$ lies in some maximal
subgroup of $L$. So $\Gamma_y < H^a$ where $a\in L$ and $H$ is of
one of the above types.

The case when $H$ is a Suzuki group over a smaller field can be
excluded  because  the essential case $\Gamma_y=H=Sz(q'^2)$,
$q^{\prime 2}\vert q^2$, cannot occur. Indeed, $Sz(q^2)$ is solvable
if and only if $q^2=2$.

\begin{lemma} \label{lem:Sz2}
With the above notation, we have $\Gamma_y \ne Sz(2).$
\end{lemma}

\begin{proof}
Assume the contrary.  We have $\Gamma_y =
Sz(2)=\left<a\right>\rtimes\left<b\right>$, $a^5 = 1$, $b^4 = 1$,
$ba b^{-1} = a^2.$ Note that $\Gamma_y$ is a normal subgroup in
$\langle g,\, y\rangle $, the subgroup $\left<a\right>$ coincides
with the derived subgroup of $\Gamma_y = Sz(2)$, and it is a
characteristic subgroup in $\Gamma_y = Sz(2)$. Since the order
$\ell$ of $g$ is prime $ > 3$, we conclude that $gag^{-1} = a$,
$yay^{-1} = a$, and hence $\gamma a \gamma^{-1} = a$ for every
$\gamma \in \left< g, y\right>$. Since $ba b^{-1} = a^2$, we have $b
\notin \left< g, y\right>$. Contradiction.
\end{proof}



\begin{lemma}\label{prop:tori}
Suppose that $\Gamma_y$ is contained in $N_L(T^a)$ or in
$N_L(\mathfrak{T}^a)$ where $a\in L$. Then $\Gamma_y$ is contained
in $T^a$ $($and thus in $B^a)$ or in $\mathfrak{T^a}$, respectively.
\end{lemma}

To prove Lemma \ref{prop:tori}, we need two more auxiliary
assertions.

\begin{sublemma} \label{sublem1}
Suppose that $\Gamma_y\leq N_L(T^a)$ $($respectively, $\Gamma_y \leq
N_L(\mathfrak{T}^a))$. Then $T_y := \Gamma_y \cap T^a $
$($respectively, $\mathfrak{T}_y :=\Gamma_y \cap \mathfrak{T}^a)$ is
$g$-invariant.
\end{sublemma}

\begin{proof}
Let $\Gamma_y\leq N_L(T^a)$. Denote by $\Gamma^2_y$ the subgroup of
$\Gamma_y$ generated by the squares of the elements of $\Gamma_y$.
It is clear that $\Gamma^2_y$ is invariant under $g$. Since
$|N_L(T)/T|=2$, all elements of $\Gamma^2_y$ belong to $T$, so
$\Gamma^2_y$ lies in $T_y$. However, all elements of $T$ are of odd
order, therefore $\Gamma^2_y=T_y$. Hence $T_y$ is invariant under
$g$. In order to get the statement for $\mathfrak{T}_y$, we repeat
the above argument with $\Gamma^2_y$ replaced by $\Gamma^4_y$.
\end{proof}

\begin{sublemma} \label{sublem2}
Suppose that $\Gamma_y$ is contained in $N_L(T^a)$ or in
$N_L(\mathfrak{T}^a)$. Then for any integer $r$ the element
$[g^{-r},x ]$ belongs to $T_y$ or to $\mathfrak{T}_y$,
respectively.
\end{sublemma}

\begin{proof}
If $\ell\vert r$, the assertion is satisfied for trivial reasons, so
assume that $r$ is prime to $\ell$. Let $\Gamma_y \leq N_L(T^a)$.
Then $z:=g^{-r}xg^{r}x \in N_L(T^a)$. Assume $z\in \Gamma_y
\setminus T_y$. Then $g^{-r}(x)=g^{-r}xg^{r}=zx$. Furthermore,
$g^{-r}$ can act only trivially on $\Gamma_y/T_y$ since for $r$
prime to $\ell$ the order of $g^{-r}$ is $\ell >3$ whereas the order
of $\Gamma_y/T_y$ is $\leq 4$. Hence $g^{-r}(z)z\min\in T_y.$ Thus
$x=(g^{-r})^\ell (x)=z^\ell xt$ with $t\in T_y$, so $z^\ell\in T_y$.
But $z\in N_L(T^a) \setminus T^a$ and $|N_L(T^a)/T^a|=2$, therefore
$z^\ell \not\in T^a$. Contradiction.  Hence $z \in T_y$. The same
proof can be given for the case $\Gamma_y \leq N_L(\mathfrak{T}^a)$.
\end{proof}

We are now ready to prove Lemma \ref{prop:tori}.

Let $a=g$, $b=xgx\min $. Then any word $w =\cdots a^k b^m a^n\cdots$
can be written as $ \cdots a^{k+m} (a^{-m}b^m)a^{n}\cdots = \cdots
g^u[g^{-m},x ]g^v\cdots$, where the commutator in the middle
belongs to $T_y$ in view of Sublemma \ref{sublem2}. Therefore $w=
g^{u_1}z_1g^{u_2}z_2 \cdots$, where $z_1, z_2\in T_y$. Since $w \in
L$, the sum $u_1 +u_2 +\cdots$ is divisible by $\ell$, so
\[
w = (g^{u_1}z_1 g^{-u_1})(g^{u_1+u_2}z_2 g^{-u_1-u_2})(
g^{u_1+u_2+u_3}\cdots )= \prod (g^{v_i}z_ig^{-v_i}).
\]
By Sublemma \ref{sublem1}, the latter element must belong to $T_y$.

The case $\Gamma_y \leq N(\mathfrak{T}^a)$ is treated in exactly the
same way. \qed

\medskip

We thus may and shall assume that $\Gamma_y\leq H^a$ where $H=B$ or
$H=\mathfrak T$.

We are now able to estimate the number of elements $y=x gx\min$,
$x\in L$, such that the group $\Gamma_y=\left<g,y\right> \cap L$ is
solvable.

Denote $\mathcal A_H := \{ H^a | \ gH^ag^{-1}=H^a,\ a\in L\}$.

Denote $L_1=\left<L,g\right>$. Note that $\left<g,y\right>\neq L_1$
because the group $\left<g,y\right>$ is solvable. Hence
$\left<g,y\right>$ is contained in a proper maximal subgroup $M$ of
$L_1$. By \cite{Kl1}, $M$ is conjugate to a subgroup of the form
$N_{L_1}(H^a)=\left<H^a,g\right>$ where $H^a$ is a maximal subgroup
in $L$ invariant under $g$.

So we may assume that $\left<g,y\right>$ lies in a semidirect
product $H^a\rtimes\left<g\right>$ and $\Gamma_y$ lies in some $H^a$
such that $gH^ag^{-1}=H^a$. We have $yH^ay^{-1}=H^a$ because $H^a$
is normal in $M$. The equality $yH^ay^{-1}=H^a$ can be rewritten as
$g(x\min H^ax )g^{-1}=x\min H^ax$, or, in other words, as $x\min
H^ax\in \mathcal A_H$. Hence it is enough to estimate, for each $H$,
the number of elements in the set
\[
S_H:=\{s\in L : sH^as\min =H^{a'}  \text{\rm{ for some }}H^a,
H^{a'}\in \mathcal A_H\}.
\]

\begin{lemma} \label{lem:sb}
\[
\vert S_B\vert \leq q^{6\ell +9}.
\]
\end{lemma}

\begin{proof} First recall that we denote $L=\ ^2B_2(q^{2\ell})$. We
have $C_L(g)=\ ^2B_2(q^2)$. The orders of the above groups are
$q^{4\ell}(q^{2\ell}-1)(q^{4\ell}+1)$ and $q^4(q^2-1)(q^4+1)$,
respectively.

We have $H=B=TU$ where $U$ is the maximal unipotent subgroup of $B$.
Denote by $C_U(g)$ the centralizer of $g$ in $U$. By the definition
of $S_H$ we have $S_B\supseteq B$ (because $B\in \mathcal A_B)$.
Furthermore, let $s\notin B$ and suppose that $sBs\min$ is
$g$-invariant. The Bruhat decomposition of $L$ contains only two
cells, hence we can represent $s$ in the form $s=u{\dot w}b$ with
$u\in U$, $b\in B$, and $w$ the non-identity element of the Weyl
group. The condition $g(sBs\min)g\min =sBs\min$ can be rewritten as
$u^{-1}guB^-u^{-1}g^{-1}u = B^-$ where $B^-$ stands for the Borel
subgroup opposite to $B$. This subgroup is invariant under $g$, so
applying $g\min$ to the last equality we conclude that
$v:=g\min(u\min)u$ normalizes $B^-$ and hence belongs to $B^-$. On
the other hand, $v$ is a product of two elements from $U$ and thus
belongs to $U$. As $B^-\cap U=1$, we conclude that $v=1$, i.e., $u$
belongs to the centralizer $C_U(g)$. Thus the set of $s\notin B$
such that $sBs\min$ is $g$-invariant is in one-to-one correspondence
with the set of pairs $\{(b,u)\}$ with $b\in B$ and $u\in C_U(g)$.
The number of such pairs equals $\vert B\vert\cdot\vert
C_U(g)\vert$. Hence the number of $s\in L$ such that $sBs\min$ is
$g$-invariant equals $\vert B \vert (1+\vert C_U(g)\vert
)=q^{4\ell}(q^{2\ell}-1)(q^4+1).$

This calculation should be repeated for every $B^a=aBa^{-1}$ such
that $g(aBa^{-1})g^{-1}=aBa^{-1}$, i.e., for each $B^a\in \mathcal
A_B$. Let us write the last condition in the form $a\min gaBa\min
g\min a=B$ and use the Bruhat decomposition $a=uwb$ for $a$, as in
the above paragraph. The same computation shows that the number of
such groups $B^a$ equals $\vert C_U(g)\vert =q^4$. Thus we conclude
that
\[
\vert S_B\vert \leq q^{4\ell }(q^{2\ell }-1)(q^4 +1)q^4 \leq
q^{6\ell +9} .
\]
\end{proof}

In order to treat the case of nonsplit tori in a similar way we need
the following

\begin{remark} \label{rem:coprime}
If $\mathfrak T_1$ and $\mathfrak T_2$ are nonconjugate nonsplit
tori in $Sz(q^{2\ell})$, then they are cyclic groups of orders
$q^{2\ell}+\sqrt{2q^{2\ell}}+1$ and $q^{2\ell}-\sqrt{2q^{2\ell}}+1$,
respectively. These are odd numbers whose difference is
$2\sqrt{2q^{2\ell}}$ which is a power of 2, hence they are coprime.
\end{remark}

\begin{lemma} \label{lem:h1}
Let $\mathfrak T$ be a nonsplit $g$-stable torus of $L=Sz(q^{2\ell})$.  Then the first
cohomology group $H^1(\left< g\right>,\mathfrak T)$ is trivial.
\end{lemma}

\begin{proof}
It is enough to prove that $H^1(\left<g\right>,\mathfrak T_\ell )=1$
where $\mathfrak T_\ell$ is a Sylow $\ell$-subgroup of $\mathfrak
T$.

Let us first prove that if $\mathfrak T_\ell\ne 1$, then $g$ does
not centralize $\mathfrak T_\ell$. Assume the contrary. Then
$\mathfrak T_\ell$ is contained in $Sz(q^2)$. The order of
$\mathfrak T$ is either $N^-=q^{2\ell}-\sqrt{2q^{2\ell}}+1$, or
$N^+=q^{2\ell}+\sqrt{2q^{2\ell}}+1$. Note that $N^+$ and $N^-$
cannot be both divisible by $\ell$ (see Remark \ref{rem:coprime}).
As $\mathfrak T_\ell < Sz(q^2)$, we have $q^4+1\equiv 0\pmod\ell$
(because the order of $\mathfrak T_\ell$ divides both $|Sz(q^2)|=
q^4(q^2-1)(q^4+1)$ and $N^+N^-=q^{4\ell}+1$) and $\ell$ does not
divide $(q^{4\ell}+1)/(q^4+1)$. On the other hand, $q^4+1\equiv
0\pmod\ell$ implies $(q^{4\ell}+1)/(q^4+1)=((q^4)^{\ell
-1}-(q^4)^{\ell -2}+\dots +1)\equiv 0\pmod\ell$. Contradiction. Thus
$g$ acts nontrivially on $\mathfrak T_\ell$.

As $\mathfrak T$ is a cyclic group, we finish the proof by noting
that if a cyclic group $C$ of the order $\ell$ acts nontrivially on
a cyclic $\ell$-group $M$ ($\ell$ odd), we have $H^i(C,M)=1$ for all
$i\ge 1$ \cite[Ch.~II, Example~7.9]{AM}.
\end{proof}

We are now able to repeat for the tori $\mathfrak T=\mathfrak T_1$
and $\mathfrak T_2$ the computations already performed for the case
of Borel subgroups.

\begin{lemma} \label{lem:st}
\[
\vert S_{\mathfrak T_1}\vert + \vert S_{\mathfrak T_2}\vert \leq
q^{3\ell + 24}.
\]
\end{lemma}

\begin{proof}
Fix a maximal nonsplit torus $\mathfrak T$ invariant under $g$. For
$s\in L$ such that $s\mathfrak Ts\min\in \mathcal A_{\mathfrak T}$,
consider $z:=g\min s\min gs$. Arguing as in the proof of Lemma
\ref{lem:sb}, we arrive at the equality $z\mathfrak Tz\min
=\mathfrak T$, i.e., $z\in N_L(\mathfrak T)$. Since $g\in
N_{L_1}(\mathfrak T)$ (recall that $L_1 = \langle L, g\rangle$), we
have $gz=s\min gs\in N_{L_1}(\mathfrak T)$, and therefore the group
$\left<g,s\min gs\right>$ is contained in $N_{L_1}(\mathfrak T)$, so
$\left<g,s\min gs\right>\cap L$ is contained in $N_L(\mathfrak T)$.
By Lemma \ref{prop:tori}, this group is contained in $\mathfrak T$,
so $z$ defines  a cocycle with values in $\mathfrak T$. By Lemma
\ref{lem:h1}, $H^1(\left<g\right>,\mathfrak T)=1$, therefore
$z=g\min t\min gt$ with $t\in \mathfrak T$. Therefore
$g(ts^{-1})=ts^{-1}$ whence $ts^{-1}=a\in C_L(g)=Sz(q^2)$. Thus
$s=a\min t$ with $a\in Sz(q^2)$, $t\in\mathfrak T$. Therefore the
number of elements $s\in L$ such that $s\mathfrak Ts\min\in \mathcal
A_{\mathfrak T}$ is bounded by $\vert\mathfrak T\vert\cdot\vert
Sz(q^2)\vert = \vert \mathfrak{T}\vert q^4(q^2-1)(q^4+1).$

This estimate should be repeated for each $\mathfrak T^\gamma\in
\mathcal A_{\mathfrak T}$, i.e., for every $\mathfrak T^\gamma$ such
that $g\mathfrak T^\gamma g^{-1}= \mathfrak T^\gamma$.  We have seen
above that $ \mathfrak T^\gamma= \gamma \mathfrak{T}\gamma^{-1}\in
\mathcal A_{\mathfrak T}$ if and only if $ \gamma \in Sz(q^2)$.
Hence the the number of  groups $\mathfrak T^{\gamma}$ is bounded by
$\vert Sz(q^2)\vert$.

Thus
\[
\begin{aligned}
\vert S_{\mathfrak T}\vert & \leq \vert \mathfrak{T}\vert\cdot
q^4(q^2-1)(q^4+1)\cdot (q^4(q^2-1)(q^4+1))\\ &={\vert
\mathfrak{T}\vert} q^8(q^2-1)^2(q^4+1)^2 \leq \vert
\mathfrak{T}\vert q^{22} .
\end{aligned}
\]
Since there are two nonconjugate nonsplit tori $\mathfrak{T}_1$ and
$\mathfrak{T}_2$, we get $\vert S_{\mathfrak T_1}\vert + \vert
S_{\mathfrak T_2}\vert \leq q^{22}( \vert\mathfrak{T}_1\vert + \vert
\mathfrak{T}_2\vert)$. Recall that the orders of the tori are
$q^{2\ell}-\sqrt{2q^{2\ell}}+1$ and $q^{2\ell}+\sqrt{2q^{2\ell}}+1$,
so each order is less than $q^{3\ell}$. Thus we get the needed
estimate $\vert S_{\mathfrak T_1}\vert + \vert S_{\mathfrak
T_2}\vert \leq 2q^{{3\ell}+22}\leq q^{{3\ell}+24}$.
\end{proof}

We can now finish the proof of Proposition \ref{lemma:field}.

By Lemmas \ref{lem:sb} and \ref{lem:st}, the total number of all
elements $s\in L$ such that $sH^as\min \in\mathcal A_H$ for some
$H^a\in \mathcal A_H$ is bounded by $q^{3\ell + 24} + q^{6\ell +
9}.$ If $\ell\geq 5$ then $6\ell + 9 \geq  3\ell +24$, and we have
$q^{3\ell + 24} + q^{6\ell + 9}\leq 2q^{6\ell + 9}\leq q^{6\ell +
11}. $

We have proved above that if the group $\langle g, xgx^{-1}\rangle$
is solvable  then $\langle g , xg x^{-1}\rangle \leq \langle H^a ,
g\rangle$ for some  $g$-stable maximal subgroup $H^a\leq L$.
Moreover, in this case $xH^ax^{-1}$ is also $g$-stable. Hence if
$\langle g , xgx^{-1}\rangle$ is solvable then there exists a
$g$-stable maximal subgroup $H^a $ such that $x H^ax^{-1}$ is also
$g$-stable. We have just estimated the number of those $x$ such that
there is a $g$-stable maximal subgroup $H^a$ for which the group
$xH^ax^{-1}$ is also $g$-stable. This number is not more than
$q^{6\ell +11}$. But
\[
q^{6\ell +11}< q^{9\ell} < \mid L\mid =
q^{4\ell}(q^{2\ell}-1)(q^{4\ell}+1).
\]
Thus we can find $x\in L$ such that the  group $\langle g,
xgx^{-1}\rangle$ is nonsolvable.

Proposition \ref{lemma:field} is proved. \qed

\medskip

Theorem \ref{prop:rank1} now follows from Propositions
\ref{lem-PSL-PSU}, \ref{lem-Suzuki-Ree}, \ref{lem-Ree-field}, and
\ref{lemma:field}. \qed


\section{General case}\label{sec:gen}

In this section we prove the main part of Theorem \ref{th:bigorder}
considering almost simple groups of Lie rank $>1$ not of type
${}^2F_4$. The case ${}^2F_4$ will be treated separately in the last
section.

\begin{theorem}\label{th:rank}
Let $L$ be a simple group of Lie type of Lie rank $\geq 2$, $L\ne
{}^2F_4(q^2)$, and let $L\leq G\leq \Aut L$. Then $G$ satisfies
$\bf{(NS)}$.
\end{theorem}

Suppose that the property {\bf (NS)} does not hold for some group
$G$. We may assume for $G$ the following property ({\bf MC} stands
for ``minimal counter-example''):

{\bf{MC}:} \nopagebreak

(a) $G$ is a finite almost simple group which does not satisfy
 {\bf(NS)};

(b) $[G,G] = L$ is a simple group of Lie type different from
${}^2F_4$;

(c) if $H$ is a group satisfying conditions (a) and (b), then the
order of $[H,H]$ is greater than or equal to the order of $L$.

\medskip

{\bf Throughout below $g \in G$ is an element of prime order $\ell
>3$ such that the group $\langle g, x gx\min\rangle$ is solvable for
every $x \in L$} (such an element exists according to hypothesis
(a)).

\medskip

Suppose that $g$ induces a field automorphism of $L$. Then one can
find a subgroup $L_1 = \langle U_{\pm \alpha}\rangle$ where $U_{\pm
\alpha}$ are root subgroups which are $g$-stable but not centralized
by $g$ (this follows from the definition of field automorphism.)
Since the order of $g$ is a prime number $\geq 5$, the group $L_2 =
L_1/Z(L_1)$ is a simple group of rank one, and $g$ induces on $L_2$
an automorphism of prime order $\geq 5$. Then the almost simple
group $G_1 = \langle g, L_2\rangle$ does not satisfy {\bf (NS)}.
This contradicts Theorem \ref{prop:rank1}.

Thus, in view of classification of automorphisms of prime order (see
Proposition \ref{prop:aut_prime}), we may assume that $g$ induces an
inner-diagonal automorphism of $L$, and therefore we may also assume
that
\[
g \in G = \langle \sigma, L\rangle
\]
where $\sigma$ is a diagonal automorphism of $L$.

\subsection{} \label{subsec5.1}
Recall that the group $L$ can be represented in the form
\[
L =[\mathbb{G}(K), \mathbb{G}(K)] =
\mathbb{G}_{\text{\rm{sc}}}(K)/Z(\mathbb{G}_{\text{\rm{sc}}}(K))
\]
where $\mathbb{G}_{\text{\rm{sc}}}$ is a simple, simply connected
linear algebraic group  defined over a finite field $K$ and $\mathbb
G=\mathbb{G}_{\text{\rm{ad}}}$ is the corresponding {\it adjoint }
group.

\begin{lemma}\label{prop:red}
There exists a reductive algebraic group $\mathfrak{G}$ defined over
a finite field $K$ satisfying the following conditions:

{\text{\rm{(i)}}} the centre of $\mathfrak{G}$ is a torus and the
derived group $\mathfrak{G}^\prime$ is simply connected;

{\text{\rm{(ii)}}} $\mathfrak{G}^\prime(K) /
Z(\mathfrak{G}^\prime(K)) \cong L$;

{\text{\rm{(iii)}}} there is $\tau \in \mathfrak{G}(K)$ such that
$\langle \tau ,\mathfrak{G}^\prime(K)\rangle /Z(\langle \tau
,\mathfrak{G}^\prime(K)\rangle \cong G.$
\end{lemma}

\begin{proof}
Let $\mathbb{H}$ be a maximal $K$-torus of $\mathbb G_{sc}$ which is
quasisplit over $K$, i.e. is contained in a $K$-defined Borel
subgroup. Further, let $\tilde{\mathbb Z} = Z(\mathbb G_{sc})$ be
the centre of $\mathbb G_{sc}$ (here we regard $\tilde{\mathbb Z}$
as a finite algebraic subgroup of $\mathbb G_{sc}$ which is also
defined over $K$ \cite{Sp}). We also identify $\tilde{\mathbb Z}$
with an algebraic subgroup of $\mathbb{H}$. Consider the embedding
$\mathfrak{i}\colon\tilde{\mathbb Z}\hookrightarrow \mathbb{H}\times
\mathbb{G}_{\text{\rm{sc}}}$ given by $\mathfrak{i}(z) = (z,
z^{-1})$. The image $\mathfrak{i}(\tilde{\mathbb Z})$ will also be
denoted by $\tilde{\mathbb Z}$.

Define the reductive group $\mathfrak{G} := (\mathbb{H}\times
\mathbb{G}_{\text{\rm{sc}}})/ \tilde{\mathbb{Z}}$. Then
$Z(\mathfrak{G}) = \mathbb{H}$, $\mathfrak{G}^\prime  =
\mathbb{G}_{\text{\rm{sc}}}$  (here we identify the groups $
\mathbb{H}$ and $\mathbb{G}_{\text{\rm{sc}}}$ with their images in
$\mathfrak{G}$). Thus we have (i), (ii).

Note that there exists an automorphism $\tilde{\sigma}$ of
$\mathbb{G}_{\text{\rm{sc}}}$ which induces the given diagonal
automorphism $\sigma$ of $L$ (because $\sigma $ is defined by its
action on the root subgroups). All such automorphisms are inner in
$\mathbb{G}_{sc}$. Thus we may assume $\tilde{\sigma} \in
\mathbb{G}_{\text{\rm{sc}}}$. Let $F$ denote the Frobenius map
naturally acting on $\mathfrak{G}$ such that
$\mathbb{G}_{\text{\rm{sc}}}^F = \mathbb{G}_{\text{\rm{sc}}}(K)$,
$\mathbb{H}^F = \mathbb{H}(K).$ Then $F(\tilde{\sigma})$ and
$\tilde{\sigma}$ induce the same automorphism of $\mathfrak{G}$
(check this on root subgroups). Hence
$F(\tilde{\sigma})\tilde{\sigma}^{-1} \in Z(\mathfrak{G}) =
\mathbb{H}.$ By Lang's theorem, $H^1(F, \mathbb{H}) = 1$, therefore
we have $F(\tilde{\sigma})\tilde{\sigma}^{-1} = F(t^{-1})t$ for some
$t \in \mathbb{H}$. Hence $\tau = \tilde{\sigma}t \in
\mathfrak{G}(K).$ This gives (iii).
\end{proof}

\begin{lemma}
We have $G \leq \mathbb{G}(K).$
\end{lemma}
\begin{proof}
The quotient $\mathfrak{G}/Z(\mathfrak{G})$ coincides with the
adjoint group $\mathbb{G}$. Let $\theta\colon
\mathfrak{G}\rightarrow \mathbb{G}$ be the natural homomorphism of
algebraic groups. We have $\theta(\mathfrak{G}(K))\leq
\mathbb{G}(K).$ Lemma \ref{prop:red} implies $G = \langle \sigma,
L\rangle \leq \theta(\mathfrak{G}(K))\leq\mathbb{G}(K).$
\end{proof}

Let $\AGL_n$, $\ASL_n$ be the algebraic groups such that
${\AGL_n}(E) = GL_n(E)$, ${\ASL_n}(E) = SL_n(E)$ for every field
$E$. Further, let $K = \mathbb F_q$, let $\overline{K}$ be an
algebraic closure of $K$, and let $\Gal (\overline{K}/K) = \langle
\tau \rangle$ where $\tau(\alpha) = \alpha^q$ for every $\alpha \in
\overline{K}$. Denote by $\sigma$ the automorphism of
$GL_n(\overline{K}),\,\,\,SL_n(\overline{K})$ given by the formula
$\sigma (A) = (A^{-1})^{t}$. The automorphism $\tau$ of
$\overline{K}$ also defines the automorphism of the matrix groups
$GL_n(\overline{K})$, $SL_n(\overline{K})$ which will be denoted by
the same symbol $\tau$. For every natural number $m$ we denote by
$F_m$ the Frobenius maps:
\[
F_m = (\sigma \tau)^m \colon
GL_n(\overline{K})\rightarrow GL_n(\overline{K}),\,\,\,
SL_n(\overline{K})\rightarrow SL_n(\overline{K}).
\]
Denote by $\AU_n(q)$, $\ASU_n(q)$ the quasisplit forms of $\AGL_n$,
$\ASL_n$ defined over  $K = \mathbb F_q$ such that
\[
\AU_n(q)(\mathbb F_{q^m}) = GL_n(\overline{K})^{F_m},\,\,\ASU_n (q)(\mathbb
F_{q^m}) =
 SL_n(\overline{K})^{F_m}.
\]

\begin{lemma} \label{lem:GLU}
Suppose that $L = A_{n-1}(q)$ or $L = {}^2A_{n-1}(q^2)$. Then in
Lemma $\ref{prop:red}$ one can take $\mathfrak{G} = \AGL_{n}$ or
$\mathfrak{G} = \AU_{n}(q)$, respectively.
\end{lemma}
\begin{proof}
For the case $L = A_{n-1}(q)$ the statement is obvious. Let $L =
{}^2A_{n-1}(q^2)$. Then $\mathbb G_{sc} = {\ASU_n (q)} =
\AU_n(q)^\prime$ and the centre of ${\ASU_n (q)}$ is a
one-dimensional anisotropic torus over $K = \mathbb F_q$. Thus we
have (i) and (ii). Note that the proof of Lemma \ref{prop:red}
implies that (iii) holds for every reductive group $\mathfrak{G}$
satisfying (i) and (ii).
\end{proof}

\subsection{} \label{subsec5.2}
The notation introduced in the next paragraph refers to the
reductive group $\mathfrak G$ and the semisimple group $\mathbb G$
as in Section \ref{subsec5.1}.

For the Chevalley groups $\mathfrak{G}(K)$ and $\mathbb G(K)$ denote
$T_\mathfrak{G} = \mathbb T_\mathfrak{G}(K)$ and  $T = \mathbb T(K)$
where $\mathbb T_\mathfrak{G}$ and $\mathbb T$ are maximal
quasisplit tori of $\mathfrak{G}$ and $\mathbb G$.  We will assume
that $\tau \in T_\mathfrak{G}$ and $\sigma \in T$. Further, for the
Chevalley group $\mathfrak{G}(K)$ (or $\mathbb{G}(K)$) there exists
the root system $\Phi$ corresponding to $T_\mathfrak{G}$ (or $T$)
which either coincides with the root system $R$ of $\mathbb G$ or is
obtained from $R$ by twisting \cite{Ca1}. We denote by
$\mathfrak{U}_\alpha$ and $U_\alpha$ the root subgroups of
$\mathfrak{G}(K)$ and $\mathbb{G}(K)$ corresponding to $\alpha \in
\Phi$. We set $U_\mathfrak{G} = \langle \mathfrak{U}_\alpha \mid
\alpha \in \Phi^+\rangle,\,\,U = \langle U_\alpha \mid \alpha \in
\Phi^+\rangle$. The groups $B_\mathfrak{G} =
T_\mathfrak{G}U_\mathfrak{G}$ and $B = TU$ (as well as all their
conjugates) are called {\it Borel subgroups} of $\mathfrak{G}(K)$
and $\mathbb{G}(K)$. Any subgroup of $\mathfrak{G}(K)$ or
$\mathbb{G}(K)$ which contains a Borel subgroup is called a
parabolic subgroup.

Fix a simple root system $\Pi$ generating $\Phi$. For $\Pi^\prime
\subset \Pi$ denote $W_{\Pi^\prime} = \langle w_\alpha \,\,\mid
\alpha \in \Pi^\prime\rangle$. Then $P_{\Pi^\prime} =
BW_{\Pi^\prime}B$ is a standard parabolic subgroup of
$\mathbb{G}(K)$. Note that every parabolic subgroup of $\mathbb
G(K)$ is conjugate to a standard parabolic subgroup by an element of
$[\mathbb{G}(K),\mathbb{G}(K)]$.

\begin{lemma} \label{lem:parab}
The element $g \in G = \langle \sigma, L\rangle \leq \mathbb G(K)$
does not belong to any proper parabolic subgroup of $\mathbb{G}(K)$.
\end{lemma}

\begin{proof}
Assume to the contrary that $g\in P$ where $P\leq \mathbb{G}(K)$ is
a proper parabolic subgroup of $\mathbb{G}(K)$.  There exists $\gamma \in
L = [\mathbb{G}(K), \mathbb{G}(K)]$ such that $\gamma
P_{\Pi^\prime}\gamma^{-1} = P$ for some $\Pi^\prime \subset \Pi$.
Since $\gamma \in L$ and since we may consider the element $g\in G =
\langle \sigma, L\rangle$ up to conjugacy in $G$, we may assume that
$P$ is the standard parabolic subgroup $P_{\Pi^\prime}$ for some
$\Pi^\prime \subset \Pi$.

Let us show that $g$ is not a unipotent element. Indeed, if $g$ is a
unipotent element, then by conjugation with some appropriate element
of $L$ we can get an element $u \in U$ having a nontrivial factor
$u_\alpha$ for some $\alpha \in \Pi$. Again we may assume $g = u =
u_\alpha v, \,\,u_\alpha\ne 1, v \in U$, and $g\in P =
P_{\Pi^\prime}$ for $\Pi^\prime = \{\alpha\}$. The image $g_1$ of
$g$ in the quotient  $L_1 = P/Z(P)R_u(P)$  is an element of prime
order $\ell > 3$. Hence $L_1$ is an almost simple group of Lie type of rank
one which does not satisfy {\bf (NS)}. This contradicts Theorem
\ref{prop:rank1}.

Let us show that $\delta g \delta^{-1}\notin T$ for every $\delta
\in \mathbb{G}(K)$. Suppose $\delta g \delta^{-1}\in  T$ for some
$\delta \in \mathbb{G}(K)$. Then we may assume $g \in T$ (the same
arguments as above). One can then find a group $G_\alpha = \langle
U_{\pm \alpha}\rangle$, $\alpha \in \Pi$, which is normalized but
not centralized by $g$. Then again we have a contradiction with
Theorem \ref{prop:rank1}.

Let now $g \in P = P_{\Pi^\prime}$, $g \notin T$, $g \notin R_u(P)$.
Then the image $g_1$ of $g$ in $ L_1 = P/Z(P)R_u(P)$ is not trivial.
Further, there exists a simple component $L_2$ of $L_1$ which is an
almost simple group of Lie type such that the component $g_2$ of
$g_1$ in $L_2$ is not trivial. Obviously, $L_3 = [L_2, L_2]$ is a
finite simple group of Lie type $\ne {}^2F_4(q^2)$ and $\vert L\vert
> \vert L_3\vert$. Since $L_2$ is a simple component of $ L_1 =
P/Z(P)R_u(P)$, the image $g_2$ of $g_1$ can be represented in the
form $g_2 = \sigma^\prime g_3$ where $g_3 \in L_3$ and
$\sigma^\prime \in L_2$ induces a diagonal automorphism of $L_3 =
[L_2, L_2]$. Then the group $G_1 = \langle \sigma^\prime,
L_3\rangle$ does not satisfy {\bf (NS)}. Hence we have a
contradiction with {\bf (MC)}.
\end{proof}

\begin{lemma} \label{lem:reg}
The element $g \in G = \langle \sigma, L\rangle \leq \mathbb
G(K)\leq  \mathbb G$ is a regular semisimple element of $\mathbb G$.
\end{lemma}

\begin{proof}
Since the order of $g$ is prime and $g$ is not unipotent, $g$ is
semisimple. Let $C_{\mathbb G}(g)$ be the centralizer of $g$ in $
\mathbb G$. This is a reductive subgroup of $\mathbb G$ \cite
[Theorem~3.5.3]{Ca2}. Suppose that $g$ is not regular. Then the
identity component $C^0_{\mathbb G}(g)$ is not a torus. Since $K$ is
a finite field, there exists a $K$-defined Borel subgroup $\mathbb
B_g$ of $C^0_{\mathbb G}(g)$. Again, the unipotent radical
$R_u(\mathbb B_g)$ is also defined and split over $K$
\cite[14.4.5]{Sp}. Hence $R_u(\mathbb B_g)(K) \ne 1$. Since
$R_u(\mathbb B_g)(K)\leq C_{\mathbb G}(g)(K) \leq \mathbb G(K)$, one
can find a nontrivial unipotent element $u \in C_{\mathbb G(K)}(g)$.
However, by Lemma \ref{lem:parab} the element $g$ does not lie in
any proper parabolic subgroup $P\leq \mathbb G (K)$, and therefore
the characteristic of $K$ does not divide the order of $C_{\mathbb
G(K)}(g)$ \cite[Proposition~6.4.5]{Ca2}. Contradiction.
\end{proof}

\begin{lemma} \label{lem:nonnorm}
The element $g$ does not normalize any unipotent subgroup $V$ of
$\mathbb G(K)$.
\end{lemma}

\begin{proof}
Assume to the contrary that $gVg^{-1} = V$ for some unipotent
subgroup $V\leq \mathbb G$. Then $V$ is a closed subgroup of
$\mathbb G$ (because it is finite). The following construction is
due to Borel and Tits \cite{BT}. Consider the sequence of subgroups
$$N_1 := N_{\mathbb G}(V), V_1 := VR_u(N_1), \dots, N_i =
N_{\mathbb G}(V_{i-1}), V_i := V_{i-1}R_u(N_i)$$ in $\mathbb G$. All
groups here are defined over $K$. Moreover, since $V$ is in a Borel
subgroup of $\mathbb G$ (indeed, $V$ belongs to a $p$-Sylow subgroup
of $G$ which is conjugate to $U\le B$), the last term $N_k = P$ is a
parabolic subgroup of $\mathbb G$ containing $N_{{\mathbb G}}(V)$,
and therefore $g \in P$ (see \cite[30.3]{Hu1}). Since $g \in \mathbb
G(K)$ and $P$ is a parabolic subgroup defined over $K$, we have $g
\in P(K)$ where $P(K)$ is a parabolic subgroup of $\mathbb G(K)$.
This is a contradiction with Lemma \ref{lem:parab}.
\end{proof}

\subsection{} \label{subsec5.3}
Recall that a Coxeter element $w_c$ of the Weyl group $W = W(\Phi)$
with respect to $\Pi$ is a product (taken in any order) of the
reflections $w_\alpha$, $\alpha \in \Pi$, where each reflection
occurs exactly once.

Let now $\mathfrak{g}$ be a preimage of $g$ in $\mathfrak{G}(K)$,
see Lemma \ref{prop:red}(iii). Since $g$ is a semisimple regular
element of $\mathbb G$ (Lemma \ref{lem:reg}), the element
$\mathfrak{g}$ is also semisimple and regular in $\mathfrak{G}$. By
\cite[\S 9]{St2} (see also \cite{GoS}), for every Coxeter element
$\bf w_c$ of $\mathfrak{G}(K)$ there exists $x \in \mathfrak{G}(K)$
such that
\[
x\mathfrak{g}x^{-1} = {\bf u}\dot{\bf w}_{c}
\]
where  ${\bf u}\in U_\mathfrak{G}$ (here $\dot {\bf w}_{c}$ is any
preimage of ${\bf w}_{c}$). We have $x = hy$ where $h \in
T_\mathfrak{G}$ and $y \in \mathfrak{G}^\prime(K)$. Then
\[
y \mathfrak{g}y^{-1} = {\bf u^\prime}\dot{\bf  w}_{c}
\]
Thus we can put the element $\mathfrak{g}$ in the Coxeter cell
$B_\mathfrak{G}\dot{\bf  w}_{c}B_\mathfrak{G}$ by conjugation with
some element from $\mathfrak{G}^\prime(K)$. So we may assume
$\mathfrak{g} \in B_\mathfrak{G}\dot{\bf w}_{c}B_\mathfrak{G}$.
Therefore we may assume $g \in B\dot{ w}_{c}B$ and moreover
\begin{equation}
g = u\dot w_c \label{eq:Cox}
\end{equation}
for some $u \in U$.

In \cite[Section 5]{GGKP2}, it was proved that for an element $g$ of
form (\ref{eq:Cox}) with an appropriate Coxeter element $w_c$,
there is $x \in L$ such that $[g,x]=u\in U$. With this choice of
$x$, put
        \[H = \langle g, x gx\min\rangle.\]
By our assumptions, $H$ is a solvable group. Since $g, u \in H$,
there is a Hall subgroup $H_{p\ell}$, where $p = \char (K)$, such
that $g\in H_{p\ell}$. Let $A$ be the maximal abelian normal
subgroup of $H_{p\ell}$, and  let $A_p$ be the $p$-Sylow subgroup of
$A$. Suppose that $A_p \ne 1$. Then $A_p$ is normalized by $g$. This
contradicts Lemma \ref{lem:nonnorm}. Hence $A_p = 1$. Then $\vert
A\vert  = \ell^s$. Let
\[
A_{[\ell ]} = \{a\in A\,\,\mid\,\,\,a^l = 1\},\,\,\,C_{A_{[\ell
]}}(g) = \{a \in A_{[\ell]}\,\,\,\mid\,\,\,gag^{-1} = a\}.
\]
We have $C_{A_{[\ell ]}}(g) \neq 1$ since any operator of order
$\ell$ acting on a vector space over the field $\mathbb F_{\ell}$ is
unipotent and hence has a nontrivial fixed point.

We have $C_{{\mathbb G}}(g) \leq N_{{\mathbb G}}(\tilde{\mathbb T})$
for some maximal torus $\tilde{\mathbb T}$ of ${\mathbb G}$ (recall
that $g$ is a regular element of ${\mathbb G}$).

Consider the group $C_{\mathbb G(K)}(g)_{[\ell ]}$ generated by all
elements of order $\ell$ in $C_{\mathbb G(K)}(g)$. Clearly,
$C_{A_{[\ell ]}}(g)\leq C_{{\mathbb G}(K)}(g)_{[\ell ]}$. Consider
three separate cases.

\medskip

{\bf Case 1.} Suppose that $C_{\mathbb G(K)}(g)_{[\ell ]} = \langle
g \rangle$.

\medskip

Then $C_{A_{[\ell ]}}(g) = C_{{\mathbb G}(K)}(g)_{[\ell ]} = \langle
g \rangle $. Since $C_{A_{[\ell ]}}(g) = \langle g \rangle $ and $A$
is abelian, we have $A_{[\ell ]}= C_{A_{[\ell ]}}(g)$. Therefore
$\langle g \rangle = A_{[\ell ]}$ is an $H_{p\ell}$-invariant
subgroup. Recall that $[g,x]=u\in H$ is unipotent. Hence there
exists a unipotent element $v \in H_{p\ell}$.  We have
\[
vgv^{-1}= g^r,  1 < r < \ell
\]
(indeed, $g$ is regular and therefore $r \ne 1$, otherwise $g$ would
commute with a unipotent element). Hence $g^{r-1} = [v, g]\in [H,
H]$ and therefore $g \in [H, H]$. On the other hand,  the generators
of the solvable group $H = \langle g, x gx^{-1}\rangle$ are not in
$[H, H]$, so  $g\notin [H, H]$. Contradiction.

\medskip

{\bf Case 2.} Suppose that $\langle g \rangle\times \langle a
\rangle \leq C_{{\mathbb G}(K)}(g)_{[\ell ]}$ for some $a\in
\tilde{\mathbb T}(K)$.

\medskip

Let $\mathfrak{L}({\mathbb G})$ and $\mathfrak{L}(\tilde{\mathbb
T})$ be the Lie algebras of ${\mathbb G}$ and $\tilde{\mathbb T}$,
respectively. Then we have a subgroup of type $\ell\times\ell$ in
$\tilde{\mathbb T}$ which acts by conjugation on the linear space
$\overline{\mathfrak{L}}=\mathfrak{L}({\mathbb G})/
\mathfrak{L}(\tilde{\mathbb T})$ defined over a field of
characteristic $p$. Since $q$ and $\ell$ are coprime, by Maschke's
theorem this action is diagonalizable. This implies that there
exists $b \in \langle a \rangle \times \langle g\rangle$ stabilizing
a nonzero vector from $\overline{\mathfrak{L}}$. Then $C_{{\mathbb
G}}(b)$ is a $K$-defined reductive subgroup of $\mathbb G$ of
nonzero semisimple rank because the Lie algebra of $C_{{\mathbb
G}}(b)$ is not equal to the Cartan subalgebra
$\mathfrak{L}(\tilde{\mathbb T})$ (see \cite[1.14]{Ca2}). The
identity component $C^0_{{\mathbb G}}(b)$ is also defined over $K$
\cite[12.1.1]{Sp}. Since $K$ is a finite field, there exists a
$K$-defined Borel subgroup of $C^0_{{\mathbb G}}(b)$. Hence the
group $C^0_{{\mathbb G}}(b)(K)$ is not a torus (see the proof of
Lemma \ref{lem:reg}). Further,
\[
g \in \tilde{\mathbb T}(K) \leq C^0_{{\mathbb G}}(b)(K)\lneqq
\mathbb G(K).
\]
Note that $g$ does not commute with unipotent elements of
$C^0_{{\mathbb G}}(b)(K)$. Then there exists a subgroup $M \leq
C^0_{{\mathbb G}}(b)(K)$, which is a Chevalley group over some
finite extension of $K$, such that $g$ normalizes $M$ but does not
centralize it and $[M,M]/Z(M)$ is a finite group of Lie type. There
exists $m \in M$ such that $m\in M/Z(M)$ induces a diagonal
automorphism of $[M,M]/Z(M)$ and $g \in \langle m ,[M,
M]/Z(M)\rangle $. The group $\langle m ,[M, M]/Z(M)\rangle $ does
not satisfy {\bf(NS)} but $\vert [M, M]/Z(M)\vert < \vert L\vert$.
This is a contradiction with {\bf (MC)}.

\medskip

{\bf Case 3.} Suppose that $\langle g \rangle\times \langle a
\rangle \leq C_{{\mathbb G}(K)}(g)_{[\ell ]}$ for some $a\notin
\tilde{\mathbb T}(K)$.

\medskip

We have $aga^{-1}=g$ in ${\mathbb G}$, and thus $a \in C_{{\mathbb
G}}(g)\leq N_{{\mathbb G}}(\tilde{\mathbb T})$. As $a\notin
\tilde{\mathbb T}(K)$, we have $a\in N_{{\mathbb G}}(\tilde{\mathbb
T})\setminus \tilde{\mathbb T}$. Let $\mathfrak{g}, \mathfrak{a},
\mathfrak{T}$ be preimages in $\mathfrak{G}$ of $g, a,
\tilde{\mathbb T}$, respectively. Since
$\mathfrak{G}/Z(\mathfrak{G})={\mathbb G}$, we have
\begin{equation}
\mathfrak{a}\mathfrak{g}\mathfrak{a}^{-1} = \mathfrak{g}\mathfrak{c}
\label{diamond}
\end{equation}
for some $\mathfrak{c}\in Z(\mathfrak{G})$. Note that
$C_{\mathfrak{G}}(\mathfrak{g}) = \mathfrak{T}$ because
$\mathfrak{g}$ is regular in $\mathfrak{G}$ and
$\mathfrak{G}^\prime$ is simply connected. Since $a \in N_{{\mathbb
G}}(\tilde{\mathbb T})\setminus \tilde{\mathbb T}$, we have
$\mathfrak{c}\ne 1$.

\begin{lemma}
Equality $(\ref{diamond})$ cannot hold except possibly for the cases
$\mathfrak{G}^\prime = {\ASL_{\ell }}$ or $\mathfrak{G}^\prime
={\ASU_{\ell }}(q)$.
\end{lemma}

\begin{proof}
As $\mathfrak{a}$ is a preimage of $a$ and $a^\ell =1$, we have
$\mathfrak{a}^\ell \in Z(\mathfrak{G})$. Hence $\mathfrak{c}^\ell
=1$. Thus $\ell$ is the order of $\mathfrak{c}$ (recall that
$\mathfrak{c}\ne 1$). Note that $\ell$ divides the order of
$Z(\mathfrak{G}^\prime)$ because $c = [\mathfrak{a},\mathfrak{g}]\in
Z(\mathfrak{G}^\prime)$. Since $\ell$ is a prime $\geq 5$, we have
$\mathfrak{G}^\prime  = \ASL_n$ or $\ASU_n(q)$ for some $n$.

Now we may assume $\mathfrak{G} = \AGL_n$ or $\mathfrak{G} =
\AU_n(q)$ (Lemma \ref{lem:GLU}).

Choose a preimage $\mathfrak{g}$ of $g$ of $\ell$-power order, say,
$\ell^s$. We have $\mathfrak{g}^\ell \in Z(\mathfrak{G}(K))$. We
have $\mathfrak{G}(\overline{K}) = GL_n(\overline{K})$. Note that
$\mathfrak{g}$ is a regular element in $ GL_n(\overline{K})$.
Therefore $n \leq \ell$ because all eigenvalues of $\mathfrak{g}$
are different and are of the form $\epsilon_{\ell
^s}\epsilon_{\ell}^m$ where $\epsilon_{\ell ^s}$ and
$\epsilon_{\ell}$ stand for fixed roots of unity of degrees $\ell
^s$ and $\ell$, respectively. Suppose that $n < \ell$.  Then the
Weyl group $W({\AGL_n})$ has no elements of order $\ell$. The
element $a$ is of order $\ell$ and, according to the hypothesis of
Case 3, belongs to $N_{{{\mathbb G}}}(\tilde{\mathbb T})\setminus
\tilde{\mathbb T}$. Since every element of
$N_{{\mathfrak{G}}}(\mathfrak{T})/\mathfrak{T}$ coincides with some
element of $W(GL_n)$ \cite[Proposition~3.3.6]{Ca2}, we have
$\mathfrak{a}\notin
N_{{\mathfrak{G}}}(\mathfrak{T})\setminus\mathfrak{T}$, and
therefore $a\notin N_{\mathbb G}(\tilde{\mathbb T})\setminus
\tilde{\mathbb T}$, contradiction with the choice of $a$. Hence
$\mathfrak{G}^\prime= {\ASL_{\ell}}$ or
$\mathfrak{G}^\prime={\ASU_{\ell}}(q)$.
\end{proof}

\begin{lemma}\label{le:glimp}
The case $\mathfrak{G}' = {\ASL_{\ell}}$ cannot occur.
\end{lemma}

\begin{proof}
We have $\mathfrak{g}^\ell \in Z(GL_{\ell}(K))$. As in the previous
lemma, we may assume that $\mathfrak{g}^{\ell ^s} = 1$ for some $s$.
Thus $\epsilon_s =\sqrt[\ell ^s]{1}\notin K$ since otherwise $g$
would be a split semisimple element of $\mathbb G$ which would
contradict Lemma \ref{lem:parab}. On the other hand, $\epsilon_{s-1}
= \sqrt[\ell ^{s-1}]{1}\in K$ since $\mathfrak{g}^\ell = \di
(\epsilon_{s-1},\epsilon_{s-1}, \ldots,\epsilon_{s-1})$. Let
$\epsilon$ be an $\ell^{th}$ root of unity. In $GL_{\ell
}(\overline{K})$ one can represent $\mathfrak{g}$ by a diagonal
matrix of the form $\di(\epsilon_s\epsilon,
\epsilon_{s}\epsilon^{2}, \ldots, \epsilon_{s}\epsilon^{\ell})$.
Clearly, $det (\mathfrak{g})=\epsilon_{s-1}$ and the characteristic
polynomial of $\mathfrak{g}$ is $x^{\ell} +
(-1)^{\ell}\epsilon_{s-1}$. The matrix $\di(\epsilon_s\epsilon,
\epsilon_{s}\epsilon^{2}, \ldots, \epsilon_{s}\epsilon^{\ell})$ is
conjugate over $\overline{K}$ to its companion matrix
\[
\mathfrak{m} =\begin{pmatrix} 0&1&0&0&0&\cdots&0\cr
0&0&1&0&0&\cdots&0\cr \cdots\cr 0&0&0&0&0&\cdots&1\cr
\epsilon_{s-1}&0&0&0&0&\cdots&0\cr\end{pmatrix} \in GL_\ell(K).
\]
Since $\mathfrak{g}$ and $\mathfrak{m}$ have the same characteristic
polynomial and $\mathfrak{g}$ is a semisimple matrix, we have
\[
\mathfrak{g} = y \mathfrak{m} y^{-1}
\]
for some $y \in GL_\ell(K)$. Further, $y = y_1d$ where $y_1 \in
SL_\ell(K)$ and $d$ is a diagonal matrix. Hence $\mathfrak{g}_1 =
y_1^{-1}\mathfrak{g}y_1$ is a monomial matrix corresponding to an
$\ell$-cycle in $W({\AGL_\ell})$. Let now $g_1$ be the image of
$\mathfrak{g}_1$ in $PGL_n(K) = \mathbb G(K)$. The element $g_1$ is
conjugate to $g$ by an element of $PSL_n(K) = L$. Then we may assume
$g_1 = g$. Let $M$ be the image in $PGL_n(K)$ of all monomial
matrices of $GL_n(K)$. Then there exists a natural epimorphism $\phi
\colon M\rightarrow S_\ell$. We have $\phi(g) \in S_\ell$. Since
$S_\ell$ satisfies condition {\bf (NS)}, so does $M$. Then there
exists $m \in M$ such that $\langle g, mgm^{-1}\rangle$ is not
solvable which is a contradiction with the choice of $g$.
\end{proof}

\begin{lemma}
The case $\mathfrak{G}' = {\ASU_{\ell}}(q)$ cannot occur.
\end{lemma}

\begin{proof}
The same arguments as in the previous lemma imply that the element
$\mathfrak{g}\in \AU_\ell(q)(K)\leq GL_\ell(K)$ is conjugate in
$GL_\ell(K)$ to the matrix
\[
\mathfrak{m} =\begin{pmatrix} 0&1&0&0&0&\cdots&0\cr
0&0&1&0&0&\cdots&0\cr \cdots\cr 0&0&0&0&0&\cdots&1\cr
\epsilon_{s-1}&0&0&0&0&\cdots&0\cr\end{pmatrix}\in \AU_\ell(q)(K) =
\mathfrak{G}(K)
\]
for some $\epsilon_{s-1}\in \sqrt[\ell^{s-1}]{1} \in E = \mathbb
F_{q^2}$ such that $\epsilon_{s-1}\epsilon_{s-1}^q = 1$. Then the
elements $\mathfrak{g}$ and $\mathfrak{m}$ are conjugate by an
element of the group $\AU_\ell(q)(K)$ \cite[Proposition~
3.7.3]{Ca2}.

Note that $\AU_\ell(q)(K) = U_\ell(E)$ is the group of unitary
matrices in $GL_\ell(E)$  where $E = \mathbb F_{q^2}$, i.e., the matrices satisfying the
condition $(\tilde{A}^{-1})^{t} = A$ where $\tilde{A}$ is the matrix
obtained from $A$ by replacing all the entries $\alpha_{ij}$ with
$\alpha_{ij}^q$.

Let $DU_\ell(E)$ be the set of diagonal unitary matrices over $E$,
and let $W_\ell \leq GL_n(K)$ be the group of monomial matrices with
nonzero entries equal to $1$. Then
\[
\mathfrak{m} \in DU_\ell(E)W_\ell \leq U_\ell(E).
\]
Further, it is easy to see that $U_\ell(E) = DU_\ell(E)SU_\ell(E)$.
Since $\mathfrak{g}$ and $\mathfrak{m}$ are conjugate by an element
of $\AU_\ell(q)(K) = U_\ell(E)$, the element $\mathfrak{g}$ is
conjugate by some element of the group $SU_\ell(E)$ to some
$\mathfrak{m}^\prime \in DU_\ell(E)W_\ell$. Thus we may assume
$\mathfrak{g} = \mathfrak{m}^\prime \in  DU_\ell(E)W_\ell$.
Moreover, the image of $\mathfrak{g}$ in  the quotient $
DU_\ell(E)W_\ell/DU_\ell(E) \cong W_\ell \cong S_\ell$ is not
trivial. Hence, as in the previous lemma, we have a contradiction
with the choice of $g$.
\end{proof}

Theorem \ref{th:rank} is proved. \qed

\section{Case ${}^2F_4$} \label{sec:F4}

In order to complete the proof of Theorem \ref{th:radel}, it remains
to consider the case of groups of type ${}^2F_4 (q^2)$.

\begin{theorem} \label{th:F4}
Let $L$ be a  group of type ${}^2F_4(q^2)$, $q^2 > 2$, and $L\leq
G\leq \Aut L$. Then $G$ satisfies {\bf{(NS)}}.
\end{theorem}

\begin{proof} Let $g\in G.$ If $g$ is a field automorphism of $L$,
then it normalizes but does not centralize a group of rank 1, and we
can use Theorem \ref{prop:rank1}.

Thus we assume that $g$ induces an inner-diagonal automorphism of
$L$. Note that every inner-diagonal automorphism is an inner
automorphism in the case $L = {}^2F_4(q^2)$. Hence $g \in L$. Further, one can define Borel and
parabolic subgroups in ${}^2F_4 (q^2)$ (see \cite{Ca2}) because
${}^2F_4 (q^2)$ has a $BN$-pair. One can also represent $L$ in the
form $\mathbb G(\overline{{\mathbb F}}_2)^{F}$ where $\mathbb G$ is
the algebraic group of type $F_4$ defined over ${\mathbb F}_2$ and
$F$ is the Frobenius map corresponding to the group ${}^2F_4 (q^2)$.
We can define ``tori'' of $L$ as groups of $F$-invariant elements of
$F$-stable tori in $\mathbb G(\overline{\mathbb F}_2)$. Denote by
$T$ the group of $F$-invariant elements of an $F$-stable quasisplit
torus. If $g \in T$, then $g$ normalizes but does not centralize a
subgroup of $L$ which is a simple group of Lie type of rank one, and
we can use Theorem \ref{prop:rank1}. We can also use Theorem
\ref{prop:rank1} in the case when $g$ belongs to a parabolic
subgroup of $L$ (see the proof of Lemma \ref{lem:parab}). If $g$
does not belong to any proper parabolic subgroup $P \leq G$, then
the order of $C_L(g)$ is odd \cite[6.4.5]{Ca2}, and therefore (see
\cite{Gow}) we can write every semisismple element $s$ (up to
conjugacy) in the form $[g,x]$ with $x \in L$.

Among maximal tori of $L$ one can find two tori $\mathfrak{T}_1$,
$\mathfrak{T}_2$ satisfying the following conditions \cite{Ma}:

(1) $\mathfrak{T}_1$ and $\mathfrak{T}_2$ are cyclic groups;

(2) $y \mathfrak{T}_1 y^{-1} \cap \mathfrak{T}_2 = 1$ for every
$y\in L$;

(3) $N_{L}(\mathfrak{T}_i)$, $i = 1,2,$ is the only maximal subgroup
of $L$ containing $\mathfrak{T}_i$;

(4) the only prime divisors of $\mid N_{L}(\mathfrak{T}_i)/
\mathfrak{T}_i\mid $ are $2$ and $3$.

In the notation of \cite{Ma}, one can take $\mathfrak{T}_1=T_{10}$
and $\mathfrak{T}_2=T_{11}$. These groups are cyclic, and their
orders are $N^-=q^4-\sqrt{2}q^3+q^2-\sqrt{2}q+1$ and
$N^+=q^4+\sqrt{2}q^3+q^2+\sqrt{2}q+1$, respectively. It is easy to
check condition (2) by showing that $N^+$ and $N^-$ are coprime (one
can see that looking at their sum and difference).

By (2), we may assume that $g$ does not belong to a torus conjugate
to one of those $\mathfrak{T}_1, \mathfrak{T}_2$, say, to
$\mathfrak{T}_1$, but $[g,x]$ is a generator of $\mathfrak{T}_1$.
Since $\ord g = \ell > 3$, condition (4) implies that $g \notin
N_{L}(\mathfrak{T}_1)$. We have $\mathfrak{T}_1 \leq H = \langle g ,
x gx\min\rangle \nleq N_{L}(\mathfrak{T}_i)$. By \cite{Ma}, we get $H =
L$, and we have property {\bf (NS)} for the group $L$.
\end{proof}

\noindent {\it Acknowledgements}. Gordeev was supported in part by
 RFBR grant N-08-01-00756-A.
Kunyavski\u\i \ and Plotkin were supported in part by the Ministry
of Absorption (Israel) and the Minerva Foundation through the Emmy
Noether Research Institute of Mathematics. A substantial part of
this work was done in MPIM (Bonn) during the visits of Kunyavski\u\i
\ and Plotkin in 2007, the visit of Plotkin in 2008, and the visit
of Gordeev in 2009. The work was discussed by all the coauthors
during the international workshops hosted by the
Heinrich-Heine-Universit\"at (D\"usseldorf) in 2007 and 2008. The
support of these institutions is highly appreciated.



\begin{thebibliography}{[BGGKPP]}

\bibitem[AM]{AM}
A. Adem, R. J. Milgram, {\it Cohomology of Finite Groups}, 2nd ed.,
Springer, Berlin et al., 2004.

\bibitem[AL]{AL}
J. Alperin, R. Lyons, {\it On conjugacy classes of $p$-elements}, J.
Algebra {\bf 19} (1971) 536--537.



\bibitem[Ba]{Ba}
R. Baer, {\it Engelsche Elemente Noetherscher Gruppen}, Math. Ann.
{\bf 133} (1957) 256--270.

\bibitem[BBGKP]{BBGKP}
T.~Bandman, M.~Borovoi, F.~Grunewald,  B.~Kunyavski\u\i ,
E.~Plotkin, {\it Engel-like characterization of radicals in finite
dimensional Lie algebras and finite groups}, Manuscr. Math. {\bf
119} (2006) 365--381.

\bibitem[BGGKPP]{BGGKPP}
T.~Bandman,  G.-M.~Greuel, F.~Grunewald, B.~Kunyavski\u\i ,
G.~Pfister, E.~Plotkin, {\it Identities for finite solvable groups
and equations in finite simple groups}, Compositio Math. {\bf 142}
(2006) 734--764.




\bibitem[BT]{BT}
A. Borel, J. Tits, {\it \'El\'ements unipotents et sous-groupes
paraboliques de groupes r\'eductifs}, I, Invent. Math. {\bf 12}
(1971) 95--104.

\bibitem[BW]{BW}
R. Brandl, J. S. Wilson, {\it Characterization of finite soluble
groups by laws in a small number of variables}, J. Algebra {\bf 116}
(1988) 334--341.

\bibitem[BWW]{BWW}
J.~N.~Bray, J.~S.~Wilson, R.~A.~Wilson, {\it A characterization of
finite soluble groups by laws in two variables}, Bull. London Math.
Soc. {\bf 37} (2005) 179--186.


\bibitem[Ca1]{Ca1}
R.~W.~Carter, {\it Simple Groups of Lie Type}, John Wiley \& Sons,
London et al., 1972.

\bibitem[Ca2]{Ca2}
R.~W.~Carter, {\it Finite Groups of Lie Type. Conjugacy Classes and
Complex Characters}, John Wiley \& Sons, Chichester et al., 1985.



\bibitem[CCNPW]{CCNPW}
J. H. Conway, R. T. Curtis, S. P. Norton, R. A. Parker,
R.~A.~Wilson, {\it Atlas of Finite Groups}, Clarendon Press, Oxford,
1985.




\bibitem[Fl1]{Fl1}
P. Flavell, {\it Finite groups in which every two elements generate
a soluble group}, Invent. Math. {\bf 121} (1995) 279--285.

\bibitem[Fl2]{Fl2}
P. Flavell, {\it A weak soluble analogue of the Baer--Suzuki
Theorem}, preprint, available on the homepage of the author at
\newline http://web.mat.bham.ac.uk/P.J.Flavell/research/preprints .

\bibitem[Fl3]{Fl3}
P. Flavell, {\it On the Fitting height of a soluble group that is
generated by a conjugacy class}, J. London Math. Soc. {\bf 66}
(2002) 101--113.


\bibitem[FGG]{FGG}
P. Flavell, S. Guest, R. Guralnick, {\it Characterizations of the
solvable radical}, submitted.






\bibitem[GGKP1]{GGKP1}
N.~Gordeev, F.~Grunewald,  B.~Kunyavski\u\i , E.~Plotkin, {\it On
the number of conjugates defining the solvable radical of a finite
group}, C. R. Acad. Sci. Paris, S\'er. I  {\bf 343} (2006) 387--392.

\bibitem[GGKP2]{GGKP2}
N.~Gordeev, F.~Grunewald,  B.~Kunyavski\u\i , E.~Plotkin, {\it A
commutator description of the solvable radical of a finite group},
Groups, Geometry, and Dynamics {\bf 2} (2008) 85--120.

\bibitem[GGKP3]{GGKP3}
N.~Gordeev, F.~Grunewald,  B.~Kunyavski\u\i ,  E.~Plotkin, {\it A
description of Baer--Suzuki type of the solvable radical of a finite
group}, J. Pure Appl. Algebra {\bf 213} (2009) 250--258.

\bibitem[GGKP4]{GGKP4}
N.~Gordeev, F.~Grunewald,  B.~Kunyavski\u\i ,  E.~Plotkin, {\it
Baer--Suzuki theorem for the solvable radical of a finite group}, C.
R. Acad. Sci. Paris, S\'er. I {\bf 347} (2009)  217--222.

\bibitem[GoS]{GoS}
N. Gordeev, J. Saxl, {\it Products of conjugacy classes in Chevalley
groups, I: Extended covering numbers}, Israel J. Math. {\bf 130}
(2002) 207--248.


\bibitem[GL]{GL}
D.~Gorenstein, R.~Lyons, {\it The Local Structure of Finite Groups
of Characteristic $2$ Type}, Mem. Amer. Math. Soc., vol.~42, Number
276, Providence, RI, 1983.

\bibitem[GLS]{GLS}
D.~Gorenstein, R.~Lyons, R.~Solomon, {\it The Classification of the
Finite Simple Groups}, Number 3, Math. Surveys and Monographs,
vol.~40, no.~3, Amer. Math. Soc., Providence, RI, 1998.

\bibitem[Gow]{Gow}
R. Gow, {\it Commutators in finite simple groups of Lie type}, Bull.
London Math. Soc. {\bf 32} (2000) 311--315.



\bibitem[Gu]{Gu}
S. Guest, {\it A solvable version of the Baer--Suzuki theorem},
Trans. Amer. Math. Soc., to appear.

\bibitem[GKPS]{GKPS}
R. Guralnick, B. Kunyavski\u\i , E. Plotkin, A. Shalev, {\it
Thompson-like characterization of radicals in groups and Lie
algebras}, J. Algebra {\bf 300} (2006) 363--375.

\bibitem[GPS]{GPS}
R. Guralnick, E. Plotkin, A. Shalev, {\it  Burnside-type problems
reated to solvability}, Internat. J. Algebra and Computation {\bf
17} (2007) 1033--1048.


\bibitem[GS]{GS}
R. M. Guralnick, J. Saxl, {\it Generation of finite almost simple
groups by conjugates}, J. Algebra {\bf 268} (2003) 519--571.



\bibitem[Hu1]{Hu1}
J. E. Humphreys, {\it Linear Algebraic Groups}, Springer-Verlag,
Berlin--Heidelberg--New York, 1981.

\bibitem[Hu2]{Hu2}
J. E. Humphreys, {\it Modular Representations of Finite Groups of
Lie Type},  London Math. Soc. Lecture Note Ser. {\bf 326}, Cambridge
Univ. Press, 2005.


\bibitem[Kl1]{Kl1}
P.~Kleidman, {\it The subgroup structure of some finite simple
groups}, Ph.D. thesis, Univ. of Cambridge, 1987.

\bibitem[Kl2]{Kl2}
P.~Kleidman, {\it The maximal subgroups of the Chevalley groups
$G_2(q)$ with $q$ odd, the Ree groups ${}^2G_2(q)$, and their
automorphism groups}, J. Algebra {\bf 117} (1988) 30--71.



\bibitem[LLS]{LLS}
R. Lawther, M. W. Liebeck, G. Seitz, {\it Fixed point ratios in
actions of finite exceptional groups of Lie type}, Pacific J. Math.
{\bf 205} (2002) 393--464.



\bibitem[LN]{LN}
V. M. Levchuk, Ya. N. Nuzhin, {\it Structure of Ree groups}, Algebra
i Logika {\bf 24} (1985), no.~1, 26--41; English transl. in Algebra
and Logic {\bf 24} (1985), no.~1, 16--26.






\bibitem[Ma]{Ma}
G. Malle, {\it The maximal subgroups of ${}^2F_4(q^2)$}, J. Algebra
{\bf 139} (1991) 52--69.



\bibitem [Ro]{Ro}
D.~J.~S.~Robinson, {\it A Course in the Theory of Groups},
Springer-Verlag, New York, 1995.



\bibitem [Sp]{Sp}
T. A. Springer, {\it Linear Algebraic Groups}, 2nd ed., Progress in
Math. {\bf 9}, Birkh\"auser, Boston, 1998.

\bibitem[SS]{SS}
T. A. Springer, R. Steinberg, {\it  Conjugacy classes}, Seminar on
Algebraic Groups and Related Finite Groups, Lecture Notes Math. {\bf
131}, Springer-Verlag, Berlin--New York, 1970, pp.~167--266.


\bibitem [St1]{St1}
R. Steinberg, {\it Lectures on Chevalley Groups}, Yale University,
1967.

\bibitem[St2]{St2}
R. Steinberg, {\it Conjugacy Classes in Algebraic Groups}, Lecture
Notes Math. {\bf 366}, Springer-Verlag, Berlin--New York, 1974.


\bibitem[Su1]{Su1}
M. Suzuki, {\it On a class of doubly transitive groups}, Ann. Math.
{\bf 75} (1962) 105--145.

\bibitem[Su2]{Su2}
M. Suzuki, {\it Finite groups in which the centralizer of any
element of order $2$ is $2$-closed}, Ann. Math. {\bf 82} (1965)
191--212.

\bibitem[Th]{Th}
J.~Thompson, {\it Non-solvable finite groups all of whose local
subgroups are solvable}, Bull. Amer. Math. Soc. {\bf 74} (1968)
383--437.



\bibitem[V]{V}
F. D. Veldkamp, {\it Roots and maximal tori in finite forms of
semisimple algebraic groups}, Math. Ann. {\bf 207} (1974) 301--314.


\bibitem[Wi]{Wi}
J. S. Wilson, {\it Characterization of the soluble radical by a
sequence of words}, preprint, 2008.





\end{thebibliography}
\end{document}